\documentclass[reqno,11pt]{amsart}
\usepackage{amssymb}
\usepackage{amsmath}
\usepackage{amsthm}
\usepackage[usenames]{color}
\usepackage{graphicx}

\usepackage{cite}
\usepackage{bbm}
\allowdisplaybreaks[4]
\usepackage{verbatim}

\usepackage[colorlinks,linkcolor=black,anchorcolor=black,citecolor=black]{hyperref}
\usepackage[margin=1.1in]{geometry} 
\usepackage{hyperref}
\usepackage{marginnote}
\usepackage{color}

\newtheorem{thm}{Theorem}[section]

\newtheorem{lem}[thm]{Lemma}
\newtheorem{prop}[thm]{Proposition}
\newtheorem{rem}{Remark}[section]

\numberwithin{equation}{section}

\newcommand{\vertiii}[1]{{\left\vert\kern-0.25ex\left\vert\kern-0.25ex\left\vert #1
		\right\vert\kern-0.25ex\right\vert\kern-0.25ex\right\vert}}

\makeatletter
\@namedef{subjclassname@2020}{%
	\textup{2020} Mathematics Subject Classification}
\makeatother

\begin{document}

\title[The compressible Navier--Stokes--Euler system] 
{Low Mach number limit of a two-phase flow model in $\mathbb{R}^3$}

\author[F. Li]{Fucai Li}
\address[FCL]{School of Mathematics, Nanjing University, 
Nanjing 210093, P. R. China}
\email{fli@nju.edu.cn}

\author[J. Ni] {Jinkai Ni}
\address[JKN]{Department of Applied Mathematics, The Hong Kong Polytechnic University, Hong Kong}
\email{jinkaini123@gmail.com}

\author[Z. Zhang]{Zhipeng Zhang}   
\address[ZPZ]{School of Mathematical Sciences, 
Ocean University of China, Qingdao 266100, P. R. China}
\email{zhangzp@ouc.edu.cn}

\author[Z. Zhang]{Zhu Zhang}
\address[ZZ]{Department of Applied Mathematics, The Hong Kong Polytechnic University, Hong Kong}
\email{zhuama.zhang@polyu.edu.hk}

\begin{abstract}
We study the simultaneous low Mach number limit of a compressible
Navier--Stokes--Euler two-phase system in $\mathbb R^3$, in which the
pressure terms in both phases are scaled by $\varepsilon^{-2}$. For
sufficiently small $H^3$-perturbations around the constant equilibrium, we
establish the global well-posedness of the scaled compressible system and derive global
energy-dissipation estimates that are uniform with respect to the Mach
number $\varepsilon$. A key feature of the analysis is the degenerate dissipation structure: viscosity acts only on 
the Navier--Stokes phase, while the drag coupling
transfers dissipation to the inviscid
Euler phase through the relaxation mode. We then prove the global well-posedness and large-time decay
of the limiting incompressible two-phase system, and show that the relative velocity $u-\omega$ decays faster than the full
velocity pair. Finally, for the  well-prepared initial data, we introduce
pressure-corrected acoustic variables to remove the singular pressure mismatch and establish a global-in-time $H^2$-error.
As a result, the scaled compressible solutions converge uniformly in time to the corresponding solution of the limiting incompressible system.
\end{abstract}

\date{\today}
	

\maketitle
\thispagestyle{empty}

\section{Introduction and main results}

\subsection{Introduction}
Two-phase flow models have attracted lots of mathematical interest in recent decades because of their physical relevance and the hyperbolic-parabolic mixed structure exhibited in the governing equations. These models arises in the study of sprays \cite{BBBDLLT-irma-2005}, sedimentation \cite{BWC-zamm-2000}, combustion \cite{Wfa-1958,Wfa-1985}, diesel engines \cite{RM-1952,RM-1952-a}, and medical biosprays \cite{BBJM-esaim-2005}.
At the kinetic level, a typical description consists of a {Vlasov--Fokker--Planck-type} equation that is coupled to a compressible Navier--Stokes system through a drag force. Under strong local alignment and noise, the particle distribution relaxes towards a local Maxwellian, and the hydrodynamic limit yields a macroscopic two-phase system coupling an isothermal Euler phase with an isentropic Navier--Stokes phase. 
A rigorous justification of this model was given in the work of Choi \cite{Choi-SIMA-2016}; related rigorous hydrodynamic limits and relative entropy arguments can be found in \cite{CCK-poincare-2016,CJ-M3AS-2021,MV08}.

In this paper, we study the following compressible Navier--Stokes--Euler two-phase flow model:%around the constant equilibrium state. After normalizing the equilibrium densities, the system takes the form
\begin{equation}\label{I1}
\left\{
\begin{aligned}
&\partial_t\rho
+{\rm div} (\rho  u )=0,
\\ 
&\partial_t(\rho  u )
+{\rm div}(\rho u \otimes u )
+ \nabla P(\rho )=\mu \Delta u+(\mu+\lambda)\nabla {\rm div}\,u- {\kappa n(u-\omega)},\\
&\partial_t n+{\rm div}(n\omega)=0,\\
&\partial_{t}(n\omega)+{\rm div}(n \omega \otimes \omega )+\nabla n= \kappa n(u-\omega).
\end{aligned}
\right.
\end{equation}
Here, for any $(t,x)\in\mathbb{R}^+\times\mathbb{R}^3$, $\rho=\rho(t,x)$ and $u=u(t,x)$ denote the density and velocity of the Navier--Stokes phase, respectively, while $n=n(t,x)$ and $\omega=\omega(t,x)$ represent the density and velocity of the Euler phase, respectively.
The pressure $P(\rho)$ satisfies a $\gamma$-law, namely, $P(\rho)=A\rho^\gamma$ with $A>0$ and $\gamma\geq1$. 
$\kappa>0$ denotes the friction coefficient, and the viscosity constants $\mu$ and $\lambda$ satisfy the usual physical condition: 
\begin{align*}\mu>0,\quad 2\mu+\lambda> 0. 
\end{align*}
The two phases are coupled through a drag force proportional to the relative $u - \omega$. It transfers dissipation from the viscous Navier--Stokes phase to the inviscid Euler phase, and plays an important role in the global stability of the sytem.  See, for example, \cite{WZT-M2AS-2023,LS-SIMA-2023}.

\subsection{Related literature}
The Cauchy problem of \eqref{I1} has been studied from several prespectives. Choi \cite{Choi-SIMA-2016} established global existence for small smooth solutions in both the whole space and the periodic domain settings, and proved exponential velocity alignment in the latter.
A key observation is that coupling to the viscous phase prevents singularity  formation for
sufficiently small perturbations, even though the uncoupled Euler dynamics usually develop singularities in 
finite time.
Subsequent studies investigate the large-time behavior of solutions: Wu, Zhang, and Zou \cite{WZZ-SIMA-2020} and Tang and Zhang \cite{TZ-JMAA-2021} obtained optimal algebraic decay rates under $L^1$-type assumptions. Zhang et al. \cite{ZWXM-ZAMP-2021} proved global existence in $\mathbb{R}^3$ under the smallness of the $H^3$-norm, with no smallness restriction on higher Sobolev norms. More recently, Li and Shou \cite{LS-SIMA-2023} studied this model in the critical Besov framework for the multi-dimensional case. Related results on the global well-posedness and low Mach number limit for other two-fluid models can be found, for example,  in \cite{YCZ-SIMA-2012,EW-SIMA-2015,EWW-ARMA-2016}.

A related line of research concerns the kinetic origin of \eqref{I1}: the Vlasov--Fokker--Planck/Navier--Stokes system. This system describes
the dynamics of particles suspended in a compressible fluid under the combined effects of drag, noise, and local alignment forces.
In the regime of strong noise and local alignment, the particle distribution function converges to a local Maxwellian,which is determined by an isothermal Euler system. Choi \cite{Choi-SIMA-2016} presented a rigorous justification of this limiting process.  Subsequently, Choi and Jung \cite{CJ-M3AS-2021} rigorously proved the hydrodynamic limit in a bounded domain with the specular reflection boundary condition for the kinetic equation and the homogeneous Dirichlet boundary condition for the Navier--Stokes velocity. Their proof combines a relative entropy method with a weak-strong stability argument.
These results provide a kinetic foundation for the Navier--Stokes--Euler two-phase system as an effective macroscopic model of fluid-particle interactions.

The present paper concerns a singular limit problem, namely the low Mach number limit of the system \eqref{I1}. This is a classical problem in the analysis of compressible flows. For the compressible Euler equations, the low Mach number limit has been investigated for both the well-prepared and ill-prepared initial data; see, for example,  the works of Ebin \cite{Ebin-CPAM-1979}, Klainerman and Majda \cite{KM-CPAM-1981,KM-CPAM-1982}, Ukai \cite{Ukai-JMKU-1986}, Schochet \cite{Schochet-CMP-1986,Schochet-JDE-1994}, and M\'etivier and Schochet \cite{MS-ARMA-2001}. For viscous fluids, the low Mach number limit of the compressible Navier--Stokes equations has been investigated by Lions and Masmoudi \cite{LM-JMPA-1998}, Danchin \cite{Danchin-2002}, Alazard \cite{AlazardARMA2006}, among many others. 

For the Navier--Stokes--Euler two phase system, the limiting dynamics depend on the pressure scaling in each phase.  Hong and Jong \cite{HJ-NARWA-2025} considered a low Mach number limit only in the Navier-Stokes phase, while the Euler phase remains compressible.  By contrast, the present paper considers simultaneous low Mach number scaling in both phases, which leads to a full incompressible two-phase system. This scaling corresponds to a regime where the macroscopic flow velocities are small compared with the characteristic acoustic speed of {\it both phases}. The main mathematical difficulty is that our scaling introduces singular acoustic terms in the Euler phase, which has no intrinsic viscous dissipation. Therefore, the necessary dissipation must be recovered from the drag coupling with the Navier--Stokes phases. Taking advantage of this dissipation mechanism, we establish a global-in-time error estimates in $H^2(\mathbb{R}^3)$, and provide a quantitative justification of the low Mach number limit to an incompressible two-phase flow model, rather than a compactness argument used in the previous literature.

\subsection{Formal derivation of the limiting system}
In this paper, we study the low Mach number limit of
Navier--Stokes--Euler system \eqref{I1} to the incompressible two-phase problem.
Let \(\varepsilon\in(0,1)\) be the Mach number.  We introduce the following scaling
\begin{align}\label{An1}
\rho^\varepsilon(t,x)
=
\rho\Big(\frac{t}{\varepsilon},x\Big),
\quad
n^\varepsilon(t,x)
=
n\Big(\frac{t}{\varepsilon},x\Big),\quad
u^\varepsilon(t,x)
=
\frac1\varepsilon
u\Big(\frac{t}{\varepsilon},x\Big),
\quad
\omega^\varepsilon(t,x)
=
\frac1\varepsilon
\omega\Big(\frac{t}{\varepsilon},x\Big).
\end{align}
The viscosity and the friction coefficients are
rescaled as follows:
\begin{align}\label{An2}
\mu^\varepsilon=\frac{\mu}{\varepsilon},
\quad
\lambda^\varepsilon=\frac{\lambda}{\varepsilon},
\quad
\kappa^\varepsilon=\frac{\kappa}{\varepsilon}.
\end{align}
Substituting the ansatzes \eqref{An1} and \eqref{An2} into \eqref{I1},  we obtain the following rescaled system:
\begin{equation}\label{I2}
\left\{
\begin{aligned}
&\partial_t\rho^\varepsilon+\operatorname{div}(\rho^\varepsilon u^\varepsilon)=0,\\
&\partial_t(\rho^\varepsilon u^\varepsilon)
+\operatorname{div}(\rho^\varepsilon u^\varepsilon\otimes u^\varepsilon)
+\frac{1}{\varepsilon^2}\nabla P(\rho^\varepsilon)
=\mu\Delta u^\varepsilon+(\mu+\lambda)\nabla\operatorname{div}u^\varepsilon
+\kappa n^\varepsilon(\omega^\varepsilon-u^\varepsilon),\\
&\partial_t n^\varepsilon+\operatorname{div}(n^\varepsilon \omega^\varepsilon)=0,\\
&\partial_t(n^\varepsilon \omega^\varepsilon)
+\operatorname{div}(n^\varepsilon \omega^\varepsilon\otimes \omega^\varepsilon)
+\frac{1}{\varepsilon^2}\nabla n^\varepsilon
=\kappa n^\varepsilon(u^\varepsilon-\omega^\varepsilon).
\end{aligned}
\right.
\end{equation}
We study the Cauchy problem for  \eqref{I2}. The initial data are given by
\begin{align*}
(\rho^\varepsilon,u^\varepsilon,n^\varepsilon,\omega^\varepsilon)|_{t = 0}=(\rho_0^\varepsilon(x),u_0^\varepsilon(x),n^\varepsilon_0(x),\omega_0^\varepsilon(x))\to (\bar{\rho},0,\bar{ n},0)\quad\text{as}\quad|x|\to\infty,
\end{align*}
where $\bar \rho$ and $\bar n$  respectively denote the prescribed constant background densities of the Navier--Stokes and Euler phases. Without loss of generality, we set $\bar\rho=\bar n = 1$ throughout this paper.

Now we the introduce perturbations $(q^\varepsilon, u^\varepsilon, n^\varepsilon, r^\varepsilon)$ around the equilibrium state $(1,0,1,0)$:
\begin{align*}
(\rho^{\varepsilon},u^\varepsilon,n^\varepsilon,\omega^\varepsilon)=(1+\varepsilon q^\varepsilon, u^\varepsilon, 1+\varepsilon r^\varepsilon, \omega^\varepsilon).
\end{align*}
Then the system \eqref{I2} can be rewritten in terms of $(q^\varepsilon, u^\varepsilon, r^\varepsilon, \omega^\varepsilon)$ as follows:
\begin{equation}\label{I3}
\left\{
\begin{aligned}
&\partial_t q^\varepsilon
+u^\varepsilon\cdot\nabla q^\varepsilon
+\frac{1+\varepsilon q^\varepsilon}{\varepsilon}
\operatorname{div}u^\varepsilon=0,\\
&\partial_tu^\varepsilon+u^\varepsilon\cdot\nabla u^\varepsilon
+\frac1\varepsilon
\frac{P'(1+\varepsilon q^\varepsilon)}
     {1+\varepsilon q^\varepsilon}
\nabla q^\varepsilon
=
\frac{\mu}{1+\varepsilon q^\varepsilon}\Delta u^\varepsilon
+\frac{\mu+\lambda}{1+\varepsilon q^\varepsilon}
\nabla\operatorname{div}u^\varepsilon
+\frac{1+\varepsilon r^\varepsilon}{1+\varepsilon q^\varepsilon}
(\omega^\varepsilon-u^\varepsilon),\\
&\partial_t r^\varepsilon
+\omega^\varepsilon\cdot\nabla r^\varepsilon
+\frac{1+\varepsilon r^\varepsilon}{\varepsilon}
\operatorname{div}\omega^\varepsilon=0,\\
&\partial_t\omega^\varepsilon+\omega^\varepsilon\cdot\nabla \omega^\varepsilon
+\frac1\varepsilon
\frac{1}{1+\varepsilon r^\varepsilon}
\nabla r^\varepsilon
=u^\varepsilon-\omega^\varepsilon,
\end{aligned}
\right.
\end{equation}
supplemented with the initial data:
\begin{align}\label{I3-1}
(q^\varepsilon,u^\varepsilon,r^\varepsilon,\omega^\varepsilon)|_{t=0}= (q_0^\varepsilon ,u_0^\varepsilon,r_0^\varepsilon,\omega^\varepsilon_0)=\Big(\frac{\rho^\varepsilon_0-1}{\varepsilon},u_0^\varepsilon, \frac{n_0^\varepsilon-1}{\varepsilon},\omega^\varepsilon_0             \Big).  
\end{align}

Now we derive the low Mach number limit of \eqref{I3}. Suppose that the velocity field $(u^\varepsilon,\omega^\varepsilon)$ has a limit $(u,\omega)$ as $\varepsilon\rightarrow 0^+.$ Then
from continuity equations \eqref{I3}$_1$ and  \eqref{I3}$_3$, we obtain 
\begin{align*}
(1+\varepsilon q^\varepsilon)\operatorname{div}u^\varepsilon
=
-\partial_t(\varepsilon q^\varepsilon)
-\varepsilon u^\varepsilon\cdot\nabla q^\varepsilon,    
\end{align*}
and
\begin{align*}
(1+\varepsilon r^\varepsilon)\operatorname{div}\omega^\varepsilon
=
-\partial_t(\varepsilon r^\varepsilon)
-\varepsilon \omega^\varepsilon\cdot\nabla r^\varepsilon.    
\end{align*}
Assume that the derivatives of $q^\varepsilon,r^\varepsilon \sim O(1)$. We can pass to the limit $\varepsilon \rightarrow 0^+$, and formally obtain the following divergence free conditions for the limit $(u,\omega)$:
\begin{align}\label{G1.4}
\operatorname{div}u = 0,\quad \operatorname{div}{\omega} = 0.
\end{align}
Moreover, the pressure terms in \eqref{I3} are in the gradient form. To see this, we define two auxiliary functions $h_1$ and $h_2$ satisfying:
\begin{align*}
h_1'(s):=\frac{P'(s)}{s},
\quad
h_2'(s):=\frac1{s}.
\end{align*}
Then the pressure terms in $\eqref{I3}_2$ and $\eqref{I3}_4$ can be written into:
\begin{align*}
 \frac1\varepsilon
\frac{P'(1+\varepsilon q^\varepsilon)}
     {1+\varepsilon q^\varepsilon}
\nabla q^\varepsilon
=
\nabla\Pi_1^\varepsilon,
\quad  \frac1\varepsilon
\frac{1}{1+\varepsilon r^\varepsilon}
\nabla r^\varepsilon
=
\nabla\Pi_2^\varepsilon,
\end{align*}
where the functions $\Pi_1^\varepsilon$ and $\Pi_2^\varepsilon$ are defined by
\begin{align*}
\Pi_1^\varepsilon
:=
\frac1{\varepsilon^2}
\big(h_1(1+\varepsilon q^\varepsilon)-h_1(1)\big),\quad
\Pi_2^\varepsilon
:=
\frac1{\varepsilon^2}
\big(h_2(1+\varepsilon r^\varepsilon)-h_2(1)\big).    
\end{align*}

Let $u\in L^2(\mathbb{R}^3)$ be a vector field. We define $L^2_\sigma(\mathbb{R}^3)=\{u\in L^2(\mathbb{R}^3)|~\text{div} u=0\}$ and the Leray projection 
$$
\mathbb P:={\rm Id}+\nabla (-\Delta )^{-1}{\rm div},
$$
which is a bounded mapping from $L^2(\mathbb{R}^3)$ to  $ L^2_\sigma(\mathbb{R}^3)$. Then the pressure terms $\nabla \Pi_1^\varepsilon$ and $\nabla \Pi_2^\varepsilon$ vanish after applying the operator $\mathbb{P}$. Hence, formally taking the limit $\varepsilon \rightarrow 0^+$ in the momentum equations \eqref{I3}$_2$ and \eqref{I3}$_4$, we obtain
\begin{align}
\partial_tu+\mathbb P(u\cdot\nabla u)-\mu\Delta u
&=
\mathbb P(\omega-u),    \label{G1.6-1}\\
\partial_t\omega+\mathbb P(\omega\cdot\nabla \omega)
&=
\mathbb P(u-\omega).    \label{G1.6}
\end{align}

Combining the equations \eqref{G1.4}--\eqref{G1.6},  we derive the following limiting incompressible two-phase flow system
\begin{equation}\label{I4}
\left\{
\begin{aligned}
&\partial_tu+u\cdot\nabla u+\nabla\pi_1-\mu\Delta u=\omega-u,\\
&\operatorname{div}u=0,\\
&\partial_t\omega+\omega\cdot\nabla \omega+\nabla\pi_2=u-\omega,\\
&\operatorname{div}\omega=0,
\end{aligned}
\right.
\end{equation}
supplemented with the initial data
\begin{align}\label{I4-1}
 (u , \omega )|_{t=0}= ( u_0, \omega _0).
\end{align}
Unlike the one-phase low Mach number scaling studied in \cite{HJ-NARWA-2025}, the double scaling studied in this work yields incompressibility on both phases and gives rise to two different limiting pressures $\pi_1$ and $\pi_2$. 

\subsection{Main results}

The main purpose of this paper is to justify the low Mach number limit from the scaled compressible system \eqref{I3}--\eqref{I3-1} to the incompressible two-phase system \eqref{I4}--\eqref{I4-1}.  The main results are stated in the order: we
first establish the uniform global well-posedness for the scaled compressible system  \eqref{I3}--\eqref{I3-1}, then obtain the well-posedness of the limiting incompressible system \eqref{I4}--\eqref{I4-1}, and finally derive the error estimates in the low Mach number limit.

\begin{thm}[Uniform regularity of the scaled two-phase flow model]\label{Th1}
Let $\varepsilon\in (0,1)$  and assume that
\[
(q_0^\varepsilon,u_0^\varepsilon,r_0^\varepsilon,\omega_0^\varepsilon)
\in H^3(\mathbb R^3),
\quad
1+\varepsilon q_0^\varepsilon>0,
\quad
1+\varepsilon r_0^\varepsilon>0.
\]
There exists a constant $\delta_0>0$, independent of $\varepsilon$, such
that, if
\begin{align}\label{TG1}
\mathcal X_{0,\varepsilon}
:=
\|(q_0^\varepsilon,u_0^\varepsilon,r_0^\varepsilon,
\omega_0^\varepsilon)\|_{H^3}^2
\leq \delta_0^2,
\end{align}
then the Cauchy problem \eqref{I3}--\eqref{I3-1} admits a unique global strong
solution
\[
(q^\varepsilon,u^\varepsilon,r^\varepsilon,\omega^\varepsilon)
\in \mathcal C([0,\infty);H^3(\mathbb R^3)).
\]
Moreover, $1+\varepsilon q^\varepsilon>0$ and
$1+\varepsilon r^\varepsilon>0$, and
\begin{align}\label{TG2}
&\sup_{t\geq0}
\|(q^\varepsilon,u^\varepsilon,r^\varepsilon,
\omega^\varepsilon)(t)\|_{H^3}^2
\nonumber\\
&\quad+
\int_0^\infty
\Big(
\|u^\varepsilon-\omega^\varepsilon\|_{H^3}^2
+\|\nabla(q^\varepsilon,r^\varepsilon,\omega^\varepsilon)\|_{H^2}^2
+\|\nabla u^\varepsilon\|_{H^3}^2
\Big)(\tau)\,{\rm d}\tau
\nonumber\\
&\qquad\leq
C_0
\|(q_0^\varepsilon,u_0^\varepsilon,r_0^\varepsilon,
\omega_0^\varepsilon)\|_{H^3}^2,
\end{align}
where $C_0>0$ is independent of $\varepsilon$.
\end{thm}

\begin{rem}
Although the Euler phase has no
intrinsic viscosity, we can still obtain the dissipation of $\omega^\varepsilon$ through the
viscous dissipation of $u^\varepsilon$ and the relaxation of
$u^\varepsilon-\omega^\varepsilon$.
\end{rem}

\begin{thm}[Well-posedness of the limiting system]\label{Th2}
Suppose that
\[
U_0:=(u_0,\omega_0)\in H^3(\mathbb R^3),
\quad
\operatorname{div}u_0=\operatorname{div}\omega_0=0.
\]
There exists a constant $\delta_1>0$ such that, if
\begin{align}\label{TGG1}
\mathcal X_0:=\|U_0\|_{H^3}^2\leq\delta_1^2,
\end{align}
then the Cauchy problem \eqref{I4}--\eqref{I4-1} admits a unique global strong
solution
\[
U:=(u,\omega)\in\mathcal C([0,\infty);H^3(\mathbb R^3))
\]
satisfying
\begin{align}\label{TGG2}
&\sup_{t\geq0}\|U(t)\|_{H^3}^2
+
\int_0^\infty
\Big(
\|u-\omega\|_{H^3}^2
+\|\nabla u\|_{H^3}^2
+\|\nabla\omega\|_{H^2}^2
\Big)(\tau)\,{\rm d}\tau
\leq C_1\mathcal X_0,
\end{align}
where $C_1>0$ is a positive constant. If, in addition, $U_0\in L^1(\mathbb R^3)$ and
$\|U_0\|_{H^3\cap L^1}^2\leq\delta_1$, then we have the following decay estimate:
\begin{align}\label{TGG3}
\|U(t)\|_{H^3}
&\leq C(1+t)^{-\frac34}\|U_0\|_{H^3\cap L^1},\\
\label{TGG5}
\|\nabla U(t)\|_{H^2}
&\leq C(1+t)^{-\frac54}\|U_0\|_{H^3\cap L^1},\\
\label{TGG4}
\|u(t)-\omega(t)\|_{H^3}
&\leq C(1+t)^{-\frac54}\|U_0\|_{H^3\cap L^1}.
\end{align}
\end{thm}

\begin{rem}
The relative velocity $u-\omega$ decays faster than the velocity pair
$U=(u,\omega)$. This enhanced decay reflects the relaxation mechanism
created by the drag coupling. 
\end{rem}

\begin{thm}[Global quantitative low Mach number limit]\label{Th3}
Let $(q^\varepsilon,u^\varepsilon,r^\varepsilon,\omega^\varepsilon)$ be the
global solution of the problem \eqref{I3}--\eqref{I3-1} constructed in Theorem \ref{Th1}, and
let $(u,\omega,\pi_1,\pi_2)$ be the corresponding global solution of the problem 
\eqref{I4}--\eqref{I4-1} constructed in Theorem \ref{Th2}. Here the pressure functions $\pi_1$ and $\pi_2$ are defined by
$$
\pi_1(t,x)=\sum_{i,j=1,\cdots 3}\mathcal{R}_i\mathcal{R}_j(u_iu_j)(t,x),\quad \pi_2(x)=\sum_{i,j=1,\cdots 3}\mathcal{R}_i\mathcal{R}_j(\omega_i\omega_j)(t,x),
$$
where the operators $\mathcal{R}_k$ with $k=1,2,3$ are the Riesz transforms. Assume that the
initial data are well-prepared in the sense that
\begin{align}\label{TD1}
\|u_0^\varepsilon-u_0\|_{H^2}
+
\|q_0^\varepsilon-\varepsilon[P'(1)]^{-1}\pi_1^0\|_{H^2}
+
\|\omega_0^\varepsilon-\omega_0\|_{H^2}
+
\|r_0^\varepsilon-\varepsilon\pi_2^0\|_{H^2}
\leq\varepsilon,
\end{align}
where $\pi_j^0(x)=\pi_j(t,x)|_{t=0}$, with $j=1,2$. Then there exists a constant $C_2>0$, independent of $\varepsilon$, such that
\begin{align}\label{TD2}
&\sup_{t\geq0}
\Big(
\|u^\varepsilon-u\|_{H^2}^2
+
\|q^\varepsilon-\varepsilon[P'(1)]^{-1}\pi_1\|_{H^2}^2
+
\|\omega^\varepsilon-\omega\|_{H^2}^2
+
\|r^\varepsilon-\varepsilon\pi_2\|_{H^2}^2
\Big)(t)
\nonumber\\
&\quad+
\int_0^\infty
\Big(
\|\nabla(u^\varepsilon-u)\|_{H^2}^2
+
\|\nabla(\omega^\varepsilon-\omega)\|_{H^1}^2
+
\|\nabla(q^\varepsilon-\varepsilon[P'(1)]^{-1}\pi_1)\|_{H^1}^2
\nonumber\\
&\qquad \qquad 
+
\|\nabla(r^\varepsilon-\varepsilon\pi_2)\|_{H^1}^2
+
\|(u^\varepsilon-\omega^\varepsilon)-(u-\omega)\|_{H^2}^2
\Big)(\tau)\,{\rm d}\tau
\leq C_2\varepsilon^2.
\end{align}
Consequently, as $\varepsilon\to0$, we have the following low-Mach-number limit result:
\begin{equation}\label{convergence1}
\left\{
\begin{aligned}
(q^\varepsilon,u^\varepsilon,r^\varepsilon,\omega^\varepsilon)
&\longrightarrow(0,u,0,\omega)
&&\text{strongly in }\mathcal C([0,\infty);H^2),\\
\varepsilon^{-2}\left(\rho^\varepsilon-1\right)
&\longrightarrow [P'(1)]^{-1}\pi_1
&&\text{strongly in }
\mathcal C([0,\infty);H^2)\cap L^\infty(\mathbb R^+;L^\infty),\\
\varepsilon^{-2}\left(n^\varepsilon-1\right)
&\longrightarrow \pi_2
&&\text{strongly in }
\mathcal C([0,\infty);H^2)\cap L^\infty(\mathbb R^+;L^\infty),\\
(\mathbb{I}-\mathbb P)u^\varepsilon
&\longrightarrow0
&&\text{strongly in }\mathcal C([0,\infty);H^2),\\
(\mathbb{I}-\mathbb P)\omega^\varepsilon
&\longrightarrow0
&&\text{strongly in }\mathcal C([0,\infty);H^2).
\end{aligned}
\right.
\end{equation}
\end{thm}

\subsection{Strategy of the proofs}

We briefly describe the main ingredients of our proofs. For Theorem \ref{Th1}, the main difficulties come from the singular acoustic terms of order $\varepsilon^{-1}$. These the singular pressure terms cannot be treated as
perturbative source terms in the higher-order energy estimates. Instead, we use the symmetrization of the system and choose suitable density-dependent weight to cancel the leading acoustic terms. The commutators can be bounded because each spacial derivative on density variables will produce an extra factor of $\varepsilon.$ Another difficulty is the degeneracy of dissipation: there is no intrinsic viscosity in the Euler phase. We instead use the relaxation from the drag term $u^\varepsilon-\omega^\varepsilon$. Combining this relaxation with the viscous dissipation of $u^\varepsilon$ yields lower-order spacial dissipation of   $\omega^\varepsilon$. We then prove these weaker dissipations are sufficient for controlling the nonlinear terms.

For Theorem \ref{Th2}, the same degenerate dissipation structure remains. To overcome this, we combine a
spectral analysis of the linearized system with a time-weighted nonlinear energy method. At low frequencies, the relaxation mode provides an additional factor for the relative velocity $u-\omega$. This makes the relative velocity has a faster decay than the full velocity pair $(u,\omega)$. The pressure functions of the limiting system are recovered from the elliptic equations obtained by taking the divergence of the two momentum equations. The resulting bounds, stated in Proposition \ref{Ppressure}, control both the pressure gradients and the material derivatives of the pressures. These pressure
estimates play an important role  in the subsequent  error estimates.

For Theorem \ref{Th3}, we introduce the pressure-corrected errors
\[
\tilde{ q}=q^\varepsilon-\varepsilon[P'(1)]^{-1}\pi_1,
\quad
\tilde{ r}=r^\varepsilon-\varepsilon\pi_2,
\quad
\tilde{ u}=u^\varepsilon-u,
\quad
\tilde{\omega}=\omega^\varepsilon-\omega.
\]
The correction extracts the incompressible pressure profiles from the density fluctuation and removes the singular pressure mismatch in
the error system. We then repeat the strategy in the Theorem \ref{Th1}: the symmetrization yields a bounds on $(\widetilde q,\widetilde u,\widetilde r,\widetilde\omega)$, the drag term provides a relaxation dissipation of $\tilde{ u}-\tilde{\omega}$, and the interactive estimates recover  
acoustic dissipations. As a result, the error functional $\mathcal{X}(t)$ defined in  \eqref{G6.2} satisfies
\[
\tilde{\mathcal X}(t)
\lesssim
\tilde{\mathcal X}(0)+\varepsilon^2
+
(\delta_0+\delta_1)\tilde{\mathcal X}(t),
\]
where $\delta_0$ and $\delta_1$ denote the sizes of initial perturbations for the compressible and limiting systems, respectively. Choosing $\delta_0$ and $\delta_1$ sufficiently small, and using the well-prepared condition \eqref{TD1}, we can obtain the global $H^2$ error estimate \eqref{TD2}.

\subsection{Outline of this paper} Section 2 introduces the
notation and analytic tools used throughout the paper. Section 3 establishes
the uniform regularity of the scaled two-phase system. In Section 4, we prove the
well-posedness of the limiting incompressible two-phase system and establish the pressure
estimates needed for the error analysis. In Section 5, we prove the global $H^2$-error estimate and
justify the quantitative low Mach number limit.

\section{Preliminaries}
\subsection{Notations}
Throughout this paper, \(C\) denotes a generic positive constant which may change from line to line. We write \(A\lesssim B\) if \(A\leq CB\) for some constant \(C>0\), and \(A\sim B\) if $C^{-1}A\leq B\leq CA$ for some constant \(C>0\). 
Unless otherwise specified, all such constants are independent of the Mach number \(\varepsilon\). Let \(X\) be a Banach space and \(I\subset\mathbb R\) be an interval.
We denote by \(\mathcal C(I;X)\) the space of continuous functions on \(I\) with values
in \(X\), and by \(L^q(I;X)\) the standard Bochner space. When \(I=(0,T)\), we use
the shorthand 
\begin{align*}
\|f\|_{L^q_T(X)}:=\|f\|_{L^q(0,T;X)}.    
\end{align*}
For $ f_1,\cdots f_k\in X$, we set
\begin{align*}
\|(f_1,\ldots,f_k)\|_X:=\sum_{j=1}^k\|f_j\|_X .    
\end{align*}
Finally, 
the Fourier transform and its inverse are denoted by
$\mathcal F f=\widehat f$ and $\mathcal F^{-1}f=\check f $, respectively.

For a multi-index \(\alpha=(\alpha_1,\alpha_2,\alpha_3)\in\mathbb N^3\), we write
\begin{align*}
\partial^\alpha=\partial_x^\alpha
:=
\partial_{x_1}^{\alpha_1}
\partial_{x_2}^{\alpha_2}
\partial_{x_3}^{\alpha_3},
\quad
|\alpha|:=\alpha_1+\alpha_2+\alpha_3 .    
\end{align*}
For simplicity, \(\partial_i\) denotes \(\partial_{x_i}\), \(i=1,2,3\).
For an integer \(m\geq0\), we use the following Sobolev-type norms:
\begin{align*}
\|g\|_{H^m}
:=
\sum_{|\alpha|\leq m}
\|\partial^\alpha g\|_{L^2_x},
\quad
\|g\|_{\dot H^m}
:=
\sum_{|\alpha|=m}
\|\partial^\alpha g\|_{L^2_x}.    
\end{align*}
Here and below, all spatial norms are taken with respect to the \(x\)-variable.
Finally,  for two operators \(\mathcal A_1\) and \(\mathcal A_2\), their commutator is
defined by
\begin{align*}
[\mathcal A_1,\mathcal A_2]
:=
\mathcal A_1\mathcal A_2-\mathcal A_2\mathcal A_1 .    
\end{align*}

\subsection{Analytic tools}
First, we introduce several Sobolev inequalities.
\begin{lem} [{\!\!\cite[Lemma 2.1]{CDM-KRM-2011}} and {\cite[Lemmas 2.1--2.2]{Dk-MZ-1992}}]\label{L2.1}   
For any {$g,h\in H^3(\mathbb{R}^3)$} and any multi-index $\alpha$  with $1\leq|\alpha|\leq3$, it holds that
\begin{align*}
\|g\|_{L^{\infty}(\mathbb{R}^3) } \lesssim&\, \|\nabla g\|_{L^{2}(\mathbb{R}^3)}^{\frac{1}{2}}
\|\nabla ^{2}g\|_{L^{2} (\mathbb{R}^3)}^{\frac{1}{2}}, \\
\|gh\|_{H^{1}(\mathbb{R}^3) }  \lesssim&\, \|g\|_{H^{2}(\mathbb{R}^3) }\|\nabla  h\|_{H^{2}(\mathbb{R}^3) }, \\
\|\partial^{\alpha}(gh)\|_{L^{2}(\mathbb{R}^3) }
\lesssim&\,\|\nabla  g\|_{H^{2}(\mathbb{R}^3) }\|\nabla  h\|_{H^{2}(\mathbb{R}^3) },\\
\|g\|_{L^6(\mathbb{R}^3)} \lesssim&\,  \|\nabla g\|_{L^2(\mathbb{R}^3)}\lesssim  \|g\|_{H^1 (\mathbb{R}^3)},\\
  \|g\|_{L^q (\mathbb{R}^3)} \lesssim&\, \|g\|_{H^1(\mathbb{R}^3) }, \quad 2\leq q\leq 6.
\end{align*}
\end{lem} 

\begin{lem}[{\!\!\cite[Theorem 1]{Stein-1970}}] \label{L2.3}
Let $0<s<3$, $1<p<q<\infty$, $\frac{1}{q}+\frac{s}{3}=\frac{1}{p}$, then
\begin{align*}
\|\Lambda^{-s} g\|_{L^q}\lesssim \|g\|_{L^p},   
\end{align*}
where $\Lambda^{-s}:=(-\Delta)^{s/2}$.
\end{lem}

Next, we recall the following commutator estimate:
\begin{lem}[{\!\!\cite[Appendix]{commutator1,commutator2}}]\label{L2.4}
For $k\geq 0$, we have
\begin{align*}
\|\nabla ^{k}(gh) \|_{L^r(\mathbb{R}^3)} \lesssim&\, \|g\|_{L^{r_1}(\mathbb{R}^3) }\|\nabla^{k} h\|_{L^{r_2} }+ \|h\|_{L^{r_3}(\mathbb{R}^3) }\|\nabla^{k}_{x}g\|_{L^{r_4} (\mathbb{R}^3)},\\
\|\nabla^{k} (gh)-g\nabla^k  h \|_{L^r(\mathbb{R}^3)} \lesssim&\, \|\nabla  g\|_{L^{r_1}(\mathbb{R}^3)}\|\nabla ^{k-1}h\|_{L^{r_2}(\mathbb{R}^3)}+ \|h\|_{L^{r_3}(\mathbb{R}^3)}\|\nabla^{k} g\|_{L^{r_4}(\mathbb{R}^3)},    
\end{align*}
where $1<r,r_2,r_4<\infty$ and $r_i(1\leq i\leq 4)$ satisfy 
\begin{align*}
\frac{1}{r_1}+\frac{1}{r_2}=\frac{1}{r_3}+\frac{1}{r_4}=\frac{1}{r}.   
\end{align*}
\end{lem}

\begin{comment}
To obtain the time-decay, we need the following convolution inequalities:

\begin{lem} [{\!\!\cite[Lemma 3.2]{CDM-KRM-2011}}]\label{L2.5}   
Given any $0<\beta_1\ne 1$ and $\beta_2>1$, it holds that
\begin{align*}
\int_0^t (1+t-s)^{-\beta_1}(1+s)^{-\beta_2} {\rm d}s \leq C(1+t)^{-\min\{\beta_1,\beta_2\}}, 
\end{align*}
for any $t\geq 0$.
\end{lem} 
\end{comment}

\begin{comment}
\begin{lem} [{\!\!\cite[Lemma 3.3]{CDM-KRM-2011}}]\label{L2.6}   
Let $\gamma>1$ and $g_1,g_2\in \mathcal{C}(\mathbb{R}_+,\mathbb{R}_+)$ with
$g_1(0)=0$. For $A\in \mathbb{R}_+$, define $\mathcal{B}_{A}:=\{y\in \mathcal{C}(\mathbb{R}_+,\mathbb{R}_+)|\, \,y\leq A+g_1(A)y+g_2(A)y^{\gamma},\,\, y(0)\leq A\}$.
Then, there exists a constant {$A_0>0$} such that for
any $0<A\leq A_0$, if $y(t) \in \mathcal{B}_A$ for each $t\geq 0$, then it holds that
\begin{align*}
 \sup_{t\geq 0}y(t)\leq 2A.
\end{align*}
\end{lem} 
\end{comment}

\section{Uniform regularity of the compressible two-phase flow model}

In this section, we prove the global existence and uniform-in-$\varepsilon$ estimates of strong solutions
to the Cauchy problem for the scaled compressible two-phase flow model \eqref{I3}--\eqref{I3-1}.

\subsection{A priori estimates}
 Let
$(q^\varepsilon,u^\varepsilon,r^\varepsilon,\omega^\varepsilon)$ be a strong
solution to the   problem  \eqref{I3}--\eqref{I3-1} in the time interval \([0,T)\) with \(T>0\).
We assume the following a priori assumption:
\begin{align}\label{G3.1}
\sup_{0\leq t<T}
\|(q^\varepsilon,u^\varepsilon,r^\varepsilon,\omega^\varepsilon)(t)\|_{H^3}
\leq \delta,
\end{align}
where the constant \(0<\delta<1\) is suitably small and independent of \(\varepsilon\).
By the Sobolev embedding $H^2(\mathbb{R}^3) \hookrightarrow L^\infty(\mathbb{R}^3)$ and the assumption \eqref{G3.1}, we have 
\begin{align*}
\frac{1}{2}\leq1+\varepsilon q ^\varepsilon \leq \frac{3}{2}, \quad \frac{1}{2}\leq 1+\varepsilon r^\varepsilon\leq \frac{3}{2}.
\end{align*}

We first establish the zero-order energy estimate under the a priori assumption \eqref{G3.1}.

\begin{lem}\label{L3.1}
Let $(q^\varepsilon, u^\varepsilon, r^\varepsilon,\omega^\varepsilon)$ be the strong solution to the {scaled} compressible Navier--Stokes--Euler system \eqref{I3}--\eqref{I3-1} over $t\in [0,T]$. Suppose that \eqref{G3.1} holds for some constant $\delta_1\in (0,1)$. There exists a positive constant $\eta_1>0$, independent of $\varepsilon$, such that
\begin{align}\label{G3.2}
&\frac{{\rm d}}{{\rm d}t} \big( P^\prime(1)\|q^\varepsilon\|_{L^2}^2+\|(u^\varepsilon,r^\varepsilon,\omega^\varepsilon)\|_{L^2}^2     \big) +\eta_1\big(  \|\nabla u^\varepsilon\|_{L^2}^2+\|{\rm div}\,u^\varepsilon\|_{L^2}^2+\|u^\varepsilon-\omega^\varepsilon\|_{L^2}^2 \big)\nonumber\\
&\quad\lesssim \delta \|\nabla (q^\varepsilon,u^\varepsilon,r^\varepsilon,\omega^\varepsilon)\|_{L^2}^2+\delta \|\nabla^2 u^\varepsilon\|_{L^2}^2+\delta\|u^\varepsilon-\omega^\varepsilon\|_{L^2}^2.
\end{align}
\end{lem}

\begin{proof}
Taking inner product of \eqref{I3}$_1$--\eqref{I3}$_4$ with
$\left(P'(1)q^\varepsilon, u^\varepsilon, r^\varepsilon,\omega^\varepsilon\right)$, we obtain  
\begin{align}\label{G3.3}
&\frac{1}{2}\frac{{\rm d}}{{\rm d}t} \left( P^\prime(1)\|q^\varepsilon\|_{L^2}^2+\|(u^\varepsilon,r^\varepsilon,\omega^\varepsilon)\|_{L^2}^2 \right)\nonumber\\
&\quad +\mu\|\nabla  u^\varepsilon\|_{L^2}^2+(\mu+\lambda) \|{\rm div}\,u^\varepsilon\|_{L^2}^2+ \|u^\varepsilon-\omega^\varepsilon\|_{L^2}^2=\sum_{i=1}^4 I_i, 
\end{align}
where
\begin{align}
&I_1=\frac{1}{2}\int_{\mathbb R^3}\left( -|q^\varepsilon|^2{\rm div}\,u^\varepsilon -|r^\varepsilon|^2{\rm div}\omega^\varepsilon +|u^\varepsilon|^2{\rm div}u^\varepsilon + |\omega^\varepsilon|^2{\rm div}\,  \omega^\varepsilon  \right){\rm d}x,\nonumber\\
&I_2=-\frac{1}{\varepsilon}\int_{\mathbb R^3} \left[\Big( \frac{P^\prime(1+\varepsilon q^\varepsilon)}{1+\varepsilon q^\varepsilon}-P^\prime(1)    \Big) \nabla q^\varepsilon \cdot u^\varepsilon +\Big( \frac{1}{1+\varepsilon r^\varepsilon}-1\Big) \nabla r^\varepsilon\cdot\omega^\varepsilon\right] {\rm d}x,\nonumber\\
&I_3=-\varepsilon\int_{\mathbb R^3} \Big(\frac{ \mu q^\varepsilon\Delta u^\varepsilon}{1+\varepsilon q^\varepsilon}+\frac{ (\mu+\lambda) q^\varepsilon\nabla {\rm div}\, u^\varepsilon}{1+\varepsilon q^\varepsilon}\Big)\cdot u^\varepsilon{\rm d}x,\nonumber\\
&I_4=\varepsilon\int_{\mathbb R^3} \frac{r^\varepsilon-q^\varepsilon}{1+\varepsilon q^\varepsilon}(\omega^\varepsilon-u^\varepsilon)\cdot \omega^\varepsilon{\rm d}x. \nonumber
\end{align}
For the term  $I_1$  we use Sobolev inequalities in Lemma \ref{L2.1} to obtain
\begin{align}\label{G3.4}
I_1\leq&\, C\|(q^\varepsilon,r^\varepsilon,u^\varepsilon,\omega^\varepsilon)\|_{L^3}\|(q^\varepsilon,r^\varepsilon,u^\varepsilon,\omega^\varepsilon)\|_{L^6}\|\nabla (u^\varepsilon,\omega^\varepsilon)\|_{L^2}  \nonumber\\
 \leq&\,  C\|(q^\varepsilon,r^\varepsilon,u^\varepsilon,\omega^\varepsilon)\|_{H^1} \|\nabla (q^\varepsilon,r^\varepsilon,u^\varepsilon,\omega^\varepsilon)\|_{L^2}^2\nonumber\\
 \leq&\,C \delta \|\nabla (q^\varepsilon,r^\varepsilon,u^\varepsilon,\omega^\varepsilon)\|_{L^2}^2.
\end{align}
For $I_2$, note that
$
\frac{1}{\varepsilon}\left|\frac{P'(1+\varepsilon q^\varepsilon)}{1+\varepsilon q^\varepsilon}-P'(1)\right|\leq C|q^\varepsilon| \text{ and } \frac{1}{\varepsilon}\left|\frac{1}{1+\varepsilon r^\varepsilon}-1\right|\leq C|r^\varepsilon|.
$
Thus, we obtain
\begin{align}\label{G3.5}
I_2\leq&\,  C\|(q^\varepsilon, r^\varepsilon)\|_{L^3} \|\nabla (q^\varepsilon,r^\varepsilon)\|_{L^2} \|(u^\varepsilon,\omega^\varepsilon)\|_{L^6}\nonumber\\
\leq&\, C\|(q^\varepsilon,r^\varepsilon)\|_{H^1}  \|\nabla (q^\varepsilon,r^\varepsilon)\|_{L^2}\|\nabla(u^\varepsilon,\omega^\varepsilon)\|_{L^2} \nonumber\\
\leq&\,C\delta \|\nabla (q^\varepsilon,r^\varepsilon,u^\varepsilon,\omega^\varepsilon)\|_{L^2}^2.
\end{align}
For the remaining terms $I_3$ and $I_4$, we use Lemma \ref{L2.1} to deduce that
\begin{align}\label{G3.6}
I_3+I_4\leq&\, C \varepsilon\|q^\varepsilon\|_{L^3} \|\nabla^2 u^\varepsilon\|_{L^2}\|u^\varepsilon\|_{L^6}+C\varepsilon \|(r^\varepsilon,q^\varepsilon)\|_{L^3}\|(u^\varepsilon,\omega^\varepsilon)\|_{L^6}\|u^\varepsilon-\omega^\varepsilon\|_{L^2}  \nonumber\\
\leq&\,C \varepsilon \|(q^\varepsilon,r^\varepsilon)\|_{H^1} \|\nabla(u^\varepsilon,\omega^\varepsilon)\|_{L^2}  (      \|\nabla^2 u^\varepsilon\|_{L^2}+\|u^\varepsilon-\omega^\varepsilon\|_{L^2})\nonumber\\
\leq&\, C\varepsilon \delta  \big(  \|\nabla(u^\varepsilon,\omega^\varepsilon)\|_{L^2}^2+\|u^\varepsilon-\omega^\varepsilon\|_{L^2}^2+\|\nabla^2 u^\varepsilon\|_{L^2}^2\big).
\end{align}
Substituting the bounds \eqref{G3.4}--\eqref{G3.6} into \eqref{G3.3} yields \eqref{G3.2}.
\end{proof}

Next, we derive the higher-order estimate of $(q^\varepsilon, u^\varepsilon, r^\varepsilon, \omega^\varepsilon)$.

\begin{lem}\label{L3.2}
Let $(q^\varepsilon, u^\varepsilon, r^\varepsilon,\omega^\varepsilon)$ be the strong solution to the scaled compressible  Navier--Stokes--Euler system \eqref{I3}--\eqref{I3-1}. Then there exists a positive constant $\eta_2$, independent of $\varepsilon$, such that
\begin{align}\label{G3.7}
&\frac{{\rm d}}{{\rm d}t} \sum_{1\leq |\alpha|\leq 3}\bigg(  \bigg\|\sqrt{\frac{P^\prime(1+\varepsilon q^\varepsilon)}{1+\varepsilon q^\varepsilon}} \partial^\alpha  q^\varepsilon\bigg\|_{L^2}^2+\Big\|\frac{\partial^\alpha r^\varepsilon}{\sqrt{1+\varepsilon r^\varepsilon}}\Big\|_{L^2}^2+\|\partial^\alpha (u^\varepsilon,\omega^\varepsilon)\|_{ L^2}^2 \bigg) \nonumber\\
&\quad
+\eta_2  \sum_{1\leq |\alpha|\leq 3}\big(\|\nabla\partial^\alpha  u^\varepsilon\|_{L^2}^2+\|\partial^\alpha(u^\varepsilon-\omega^\varepsilon)\|_{ L^2}^2 + \|{\rm div}\,\partial^\alpha u^\varepsilon\|_{L^2}^2\big)
\lesssim  {\delta}\|\nabla (q^\varepsilon,u^\varepsilon,r^\varepsilon,\omega^\varepsilon)\|_{H^2}^2.
\end{align}   
\end{lem}

\begin{proof}
Applying \(\partial^\alpha\) to \eqref{I3}$_1$--\eqref{I3}$_4$, where \(1\leq |\alpha|\leq3\), we obtain
\begin{equation}\label{G3.8}
\left\{
\begin{aligned}
&\partial_t\partial^\alpha q^\varepsilon
+u^\varepsilon\cdot\nabla\partial^\alpha q^\varepsilon
+\frac{1+\varepsilon q^\varepsilon}{\varepsilon}
\operatorname{div}\partial^\alpha u^\varepsilon
=
-[\partial^\alpha,u^\varepsilon\cdot\nabla]q^\varepsilon
-[\partial^\alpha,q^\varepsilon\operatorname{div}]u^\varepsilon,\\
&\partial_t\partial^\alpha u^\varepsilon
+u^\varepsilon\cdot\nabla\partial^\alpha u^\varepsilon
+\frac1{\varepsilon}
\frac{P'(1+\varepsilon q^\varepsilon)}
     {1+\varepsilon q^\varepsilon}
\nabla\partial^\alpha q^\varepsilon
-\mu\Delta\partial^\alpha u^\varepsilon
-(\mu+\lambda)\nabla\operatorname{div}\partial^\alpha u^\varepsilon
-\partial^\alpha(\omega^\varepsilon-u^\varepsilon)
\\
&\quad
=
-[\partial^\alpha,u^\varepsilon\cdot\nabla]u^\varepsilon
-\frac1{\varepsilon}
\Big[
\partial^\alpha,
\frac{P'(1+\varepsilon q^\varepsilon)}
     {1+\varepsilon q^\varepsilon}
\Big]\nabla q^\varepsilon+\partial^\alpha\Big(\frac{\varepsilon(r^\varepsilon-q^\varepsilon)}
      {1+\varepsilon q^\varepsilon}
(\omega^\varepsilon-u^\varepsilon)
\Big)
\\
&\qquad\,
-\partial^\alpha\Big(
\frac{\mu\varepsilon q^\varepsilon}{1+\varepsilon q^\varepsilon}
\Delta u^\varepsilon
+\frac{(\mu+\lambda)\varepsilon q^\varepsilon}
      {1+\varepsilon q^\varepsilon}
\nabla\operatorname{div}u^\varepsilon
\Big),\\ 
&\partial_t\partial^\alpha r^\varepsilon
+\omega^\varepsilon\cdot\nabla\partial^\alpha r^\varepsilon
+\frac{1+\varepsilon r^\varepsilon}{\varepsilon}
\operatorname{div}\partial^\alpha\omega^\varepsilon
=-[\partial^\alpha,\omega^\varepsilon\cdot\nabla]r^\varepsilon
-[\partial^\alpha,r^\varepsilon\operatorname{div}]\omega^\varepsilon,
\\ 
&\partial_t\partial^\alpha\omega^\varepsilon
+\omega^\varepsilon\cdot\nabla\partial^\alpha\omega^\varepsilon
+\frac1{\varepsilon}
\frac1{1+\varepsilon r^\varepsilon}
\nabla\partial^\alpha r^\varepsilon
-\partial^\alpha(u^\varepsilon-\omega^\varepsilon)\\
&\quad=
-[\partial^\alpha,\omega^\varepsilon\cdot\nabla]\omega^\varepsilon
-\frac1{\varepsilon}
\Big[
\partial^\alpha,
\frac1{1+\varepsilon r^\varepsilon}
\Big]\nabla r^\varepsilon,
\end{aligned}
\right.
\end{equation}
where we have used the following identities:
\begin{align*}
\frac{1}{\varepsilon}    [\partial^\alpha ,(1+\varepsilon q^\varepsilon){\rm div}] u^\varepsilon= [\partial^\alpha ,  q^\varepsilon {\rm div}] u^\varepsilon, \quad
\frac{1}{\varepsilon}    [\partial^\alpha ,(1+\varepsilon r^\varepsilon){\rm div} ]  \omega^\varepsilon= [\partial^\alpha ,  r^\varepsilon {\rm div}]\omega^\varepsilon.
\end{align*}

Taking inner product of \eqref{G3.8}$_1$--\eqref{G3.8}$_4$ with
$\big(\frac{P'(1+\varepsilon q^\varepsilon)}
     {1+\varepsilon q^\varepsilon}
\partial^\alpha q^\varepsilon, \partial^\alpha u^\varepsilon, \frac{\partial^\alpha r^\varepsilon}{1+\varepsilon r^\varepsilon},\partial^\alpha\omega^\varepsilon\big)$ respectively,  we obtain

\begin{align}\label{G3.9}
&\frac12\frac{{\rm d}}{{\rm d}t}
\bigg(
\Big\|
\sqrt{\frac{P'(1+\varepsilon q^\varepsilon)}
     {1+\varepsilon q^\varepsilon}}
\partial^\alpha q^\varepsilon
\Big\|_{L^2}^2
+\|\partial^\alpha u^\varepsilon\|_{L^2}^2
+\Big\|
\frac{\partial^\alpha r^\varepsilon}{\sqrt{1+\varepsilon r^\varepsilon}}
\Big\|_{L^2}^2
+\|\partial^\alpha\omega^\varepsilon\|_{L^2}^2
\bigg)
\nonumber\\
&\quad +\mu\|\nabla\partial^\alpha u^\varepsilon\|_{L^2}^2
+(\mu+\lambda)\|\operatorname{div}\partial^\alpha u^\varepsilon\|_{L^2}^2
+\|\partial^\alpha(u^\varepsilon-\omega^\varepsilon)\|_{L^2}^2=\sum_{i=1}^6J_i,
\end{align}
where 
\begin{align}
&J_1=\frac12\int_{\mathbb R^3}
\partial_t\Big(
\frac{P'(1+\varepsilon q^\varepsilon)}
     {1+\varepsilon q^\varepsilon}
\Big)
|\partial^\alpha q^\varepsilon|^2 {\rm d}x
+
\frac12\int_{\mathbb R^3}
\partial_t\Big(
\frac{1}{1+\varepsilon r^\varepsilon}
\Big)
|\partial^\alpha r^\varepsilon|^2 {\rm d}x,
\nonumber\\
&J_2=-
\int_{\mathbb R^3}
\frac{P'(1+\varepsilon q^\varepsilon)}
     {1+\varepsilon q^\varepsilon}
\Big(
[\partial^\alpha,u^\varepsilon\cdot\nabla]q^\varepsilon
+
[\partial^\alpha,q^\varepsilon\operatorname{div}]u^\varepsilon
\Big)
\partial^\alpha q^\varepsilon{\rm d}x-\int_{\mathbb{R}^3}[\partial^\alpha,u^\varepsilon\cdot \nabla ]u^\varepsilon\cdot \partial^\alpha u^\varepsilon {\rm d}x
\nonumber\\
&\qquad-
\int_{\mathbb R^3}
\frac{1}{1+\varepsilon r^\varepsilon}
\Big(
[\partial^\alpha,\omega^\varepsilon\cdot\nabla]r^\varepsilon
+
[\partial^\alpha,r^\varepsilon\operatorname{div}]\omega^\varepsilon
\Big)
\partial^\alpha r^\varepsilon{\rm d}x-\int_{\mathbb{R}^3}[\partial^\alpha,\omega^\varepsilon\cdot\nabla]\omega^\varepsilon\cdot \partial^\alpha \omega^\varepsilon {\rm d}x,
\nonumber\\
&J_3=
\frac12\int_{\mathbb R^3}
\operatorname{div}\Big(
\frac{P'(1+\varepsilon q^\varepsilon)}
     {1+\varepsilon q^\varepsilon}
u^\varepsilon
\Big)
|\partial^\alpha q^\varepsilon|^2{\rm d}x
+
\frac12\int_{\mathbb R^3}
\operatorname{div}\Big(
\frac{\omega^\varepsilon}{1+\varepsilon r^\varepsilon}
\Big)
|\partial^\alpha r^\varepsilon|^2 {\rm d}x
\nonumber\\
&\qquad+\frac12\int_{\mathbb R^3}
|\partial^\alpha u^\varepsilon|^2
\operatorname{div}u^\varepsilon {\rm d}x
+
\frac12\int_{\mathbb R^3}
|\partial^\alpha\omega^\varepsilon|^2
\operatorname{div}\omega^\varepsilon {\rm d}x,\nonumber\\
&J_4=
\frac1{\varepsilon}\int_{\mathbb R^3}
\nabla\Big(
\frac{P'(1+\varepsilon q^\varepsilon)}
     {1+\varepsilon q^\varepsilon}
\Big)
\partial^\alpha q^\varepsilon\cdot\partial^\alpha u^\varepsilon {\rm d}x
+
\frac1{\varepsilon}\int_{\mathbb R^3}
\nabla\Big(
\frac{1}{1+\varepsilon r^\varepsilon}
\Big)
\partial^\alpha r^\varepsilon\cdot\partial^\alpha\omega^\varepsilon {\rm d}x
\nonumber\\
&\qquad-
\frac1{\varepsilon}\int_{\mathbb R^3}
\Big[
\partial^\alpha,
\frac{P'(1+\varepsilon q^\varepsilon)}
     {1+\varepsilon q^\varepsilon}
\Big]\nabla q^\varepsilon
\cdot\partial^\alpha u^\varepsilon {\rm d}x
-\frac1{\varepsilon}
\int_{\mathbb R^3}
\Big[
\partial^\alpha,
\frac1{1+\varepsilon r^\varepsilon}
\Big]\nabla r^\varepsilon
\cdot\partial^\alpha\omega^\varepsilon {\rm d}x,
\nonumber\\
&J_5=
\int_{\mathbb R^3}
\partial^\alpha\Big(
\frac{\varepsilon(r^\varepsilon-q^\varepsilon)}
     {1+\varepsilon q^\varepsilon}
(\omega^\varepsilon-u^\varepsilon)
\Big)\cdot\partial^\alpha u^\varepsilon {\rm d}x,\nonumber\\
&J_6=
\varepsilon\int_{\mathbb R^3}
\partial^\alpha\Big(
\frac{\mu q^\varepsilon}
     {1+\varepsilon q^\varepsilon}
\Delta u^\varepsilon
+
\frac{(\mu+\lambda)  q^\varepsilon}
     {1+\varepsilon q^\varepsilon}
\nabla\operatorname{div}u^\varepsilon
\Big)\cdot\partial^\alpha u^\varepsilon {\rm d}x.\nonumber
\end{align}

We now estimate $J_1-J_6$ term by term. For $J_1$, note that
\begin{align}\label{G3.10-1}
\left|\partial_t\left(\frac{P'(1+\varepsilon q^\varepsilon)}{1+\varepsilon q^\varepsilon}\right)\right|\leq C\varepsilon|\partial_tq^\varepsilon|,\quad  \left|\partial_t\left(\frac{1}{1+\varepsilon r^\varepsilon}\right)\right|\leq C\varepsilon|\partial_tr^\varepsilon|,
\end{align}
and
\begin{align}\label{G3.10}
|\partial_tq^\varepsilon|+|\partial_tr^\varepsilon |\lesssim&\, \frac{1}{\varepsilon}(1+\|(q^\varepsilon,r^\varepsilon)\|_{L^\infty})\|\nabla (u^\varepsilon,\omega^\varepsilon)\|_{L^\infty}+\|(u^\varepsilon,\omega^\varepsilon)\|_{L^\infty} \|\nabla (q^\varepsilon,r^\varepsilon)\|_{L^\infty}\nonumber\\
\lesssim&\, \Big(\frac{1}{\varepsilon}+1\Big) \|\nabla(q^\varepsilon,u^\varepsilon,r^\varepsilon,\omega^\varepsilon)\|_{H^2}\lesssim \frac{\delta}{\varepsilon}.
\end{align}
where we have used  the mass equations \eqref{I3}$_1$ and \eqref{I3}$_3$ in \eqref{G3.10}. Then substituting \eqref{G3.10-1} and \eqref{G3.10} into $J_1$, and using Lemma \ref{L2.1}, we deduce that
\begin{align}\label{G3.11}
J_{1}\lesssim&\,  \varepsilon \|\partial_{t}q^\varepsilon\|_{L^\infty}\|\partial^\alpha q^\varepsilon\|_{L^2}^2+\varepsilon \|\partial_{t}r
^\varepsilon\|_{L^\infty}\|\partial^\alpha r^\varepsilon\|_{L^2}^2  \nonumber\\
\lesssim&\, \delta \|\nabla (q^\varepsilon,r^\varepsilon)\|_{H^2}^2.
\end{align}
For $J_2$ involving commutators,  we use Lemmas \ref{L2.1} and \ref{L2.4} to obtain
\begin{align}\label{G3.12}
J_2\lesssim&\,  (1+\|q^\varepsilon\|_{L^\infty})(\|[\partial^\alpha,u^\varepsilon\cdot\nabla]q^\varepsilon\|_{L^2}+\|[\partial^\alpha,q^\varepsilon {\rm div}]u^\varepsilon\|_{L^2})\|\partial^\alpha q^\varepsilon\|_{L^2} \nonumber\\
&+\|\nabla (q^\varepsilon,u^\varepsilon,r^\varepsilon,\omega^\varepsilon)\|_{L^\infty}(1+ \|(q^\varepsilon,u^\varepsilon,r^\varepsilon,\omega^\varepsilon)\|_{L^\infty}) \|(\partial^\alpha q^\varepsilon,\partial^\alpha r^\varepsilon)\|_{L^2}^2\nonumber\\
&+(\|[\partial^\alpha,\omega^\varepsilon\cdot\nabla]r^\varepsilon\|_{L^2}+\|[\partial^\alpha,r^\varepsilon {\rm div}]\omega^\varepsilon\|_{L^2})\|\partial^\alpha r^\varepsilon\|_{L^2}\nonumber\\
&+\|[\partial^\alpha,u^\varepsilon\cdot\nabla u^\varepsilon]\|_{L^2}\|\partial^\alpha u^\varepsilon\|_{L^2}+\|[\partial^\alpha,\omega^\varepsilon\cdot\nabla \omega^\varepsilon]\|_{L^2}\|\partial^\alpha \omega^\varepsilon\|_{L^2}\nonumber\\
\lesssim&\, (1+\|(q^\varepsilon,u^\varepsilon,r^\varepsilon,\omega^\varepsilon)\|_{H^3})\|(q^\varepsilon,u^\varepsilon,r^\varepsilon,\omega^\varepsilon)\|_{H^3} \|\nabla (q^\varepsilon,u^\varepsilon,r^\varepsilon,\omega^\varepsilon)\|_{H^2}^2\nonumber\\
\lesssim&\, \delta  \|\nabla (q^\varepsilon,u^\varepsilon,r^\varepsilon,\omega^\varepsilon)\|_{H^2}^2.
\end{align}

By a direct computation, we deduce that
\begin{align}\label{G3.14}
J_3+J_5\lesssim &\, \sum_{1\leq |\alpha|\leq3}\|\nabla(u^\varepsilon,\omega^\varepsilon)\|_{L^\infty} \|(\partial^\alpha u^\varepsilon,\partial^\alpha\omega^\varepsilon)\|_{L^2}^2\nonumber\\
&+ \varepsilon\|\nabla u^\varepsilon\|_{H^2}\|\nabla (\omega^\varepsilon-u^\varepsilon)\|_{H^2} \Big\| \nabla\Big(\frac{r^\varepsilon-q^\varepsilon}{1+\varepsilon q^\varepsilon}\Big)    \Big\|_{H^2} \nonumber\\
\lesssim&\, \delta \|\nabla (q^\varepsilon,u^\varepsilon,r^\varepsilon,\omega^\varepsilon)\|_{H^2}^2.
\end{align}

For  \(J_{4}\) involving singular factor $\frac1\varepsilon,$ we observe that 
$$
\frac{1}{\varepsilon}\left|\nabla\left(\frac{P^\prime(1+\varepsilon q^\varepsilon)}{1+\varepsilon q^\varepsilon}\right)\right|+\frac{1}{\varepsilon}\left|\nabla \left(\frac{1}{1+\varepsilon r^\varepsilon}\right)\right|\leq C|\nabla(q^\varepsilon,r^\varepsilon)|.$$
Then it holds that
\begin{align}\label{G3.13}
J_{4}\lesssim&\,   \|\nabla (q^\varepsilon,r^\varepsilon)\|_{L^\infty} \|(\partial^\alpha q^\varepsilon,\partial^\alpha r^\varepsilon)\|_{L^2} \|(\partial^\alpha u^\varepsilon,\partial^\alpha\omega^\varepsilon)\|_{L^2} +\|\nabla (q^\varepsilon,r^\varepsilon)\|_{H^2}^2\|(u^\varepsilon,\omega^\varepsilon)\|_{H^3}  \nonumber\\
\lesssim&\, \delta \|\nabla(q^\varepsilon,u^\varepsilon,r^\varepsilon,\omega^\varepsilon)\|_{H^2}^2.
\end{align}
For \(J_{6}\) involving viscous term, we split it into  
 \(J_{6}=J_{6,1}+J_{6,2}\), where
\begin{align*}
J_{6,1}
&:=-\varepsilon\mu
\int_{\mathbb R^3}
\partial^\alpha\Big(
\frac{q^\varepsilon}{1+\varepsilon q^\varepsilon}
\Delta u^\varepsilon
\Big)\cdot \partial^\alpha u^\varepsilon {\rm d}x,    \\
J_{6,2}
&:=-\varepsilon(\mu+\lambda)
\int_{\mathbb R^3}
\partial^\alpha\Big(
\frac{q^\varepsilon}{1+\varepsilon q^\varepsilon}
\nabla\operatorname{div}u^\varepsilon
\Big)\cdot \partial^\alpha u^\varepsilon {\rm d}x .    
\end{align*}
Note that
\begin{align*}
\Big\|
\frac{q^\varepsilon}{1+\varepsilon q^\varepsilon}
\Big\|_{H^3}
+
\Big\|
\nabla\Big(
\frac{q^\varepsilon}{1+\varepsilon q^\varepsilon}
\Big)
\Big\|_{L^3\cap L^\infty}
\lesssim
\|q^\varepsilon\|_{H^3}
\lesssim \delta .    
\end{align*}
Then integrating by parts, we deduce that
\begin{align}
J_{6,1}
=&\varepsilon\mu
\int_{\mathbb R^3}
\frac{q^\varepsilon}{1+\varepsilon q^\varepsilon}
|\nabla\partial^\alpha u^\varepsilon|^2{\rm d}x
+\varepsilon\mu
\int_{\mathbb R^3}
\nabla\Big(
\frac{q^\varepsilon}{1+\varepsilon q^\varepsilon}
\Big)
\cdot\nabla\partial^\alpha u^\varepsilon\,\cdot
\partial^\alpha u^\varepsilon {\rm d}x\nonumber
\\
&-\varepsilon\mu
\int_{\mathbb R^3}
\Big[
\partial^\alpha,
\frac{q^\varepsilon}{1+\varepsilon q^\varepsilon}
\Big]\Delta u^\varepsilon
\cdot\partial^\alpha u^\varepsilon {\rm d}x \nonumber\\
\lesssim&
\varepsilon\delta
\|\nabla\partial^\alpha u^\varepsilon\|_{L^2}^2
+
\varepsilon
\Big\|
\nabla\Big(
\frac{q^\varepsilon}{1+\varepsilon q^\varepsilon}
\Big)
\Big\|_{L^3}
\|\nabla\partial^\alpha u^\varepsilon\|_{L^2}
\|\partial^\alpha u^\varepsilon\|_{L^6}\nonumber\\
&+\varepsilon
\Big\|
\Big[
\partial^\alpha,
\frac{q^\varepsilon}{1+\varepsilon q^\varepsilon}
\Big]\Delta u^\varepsilon
\Big\|_{L^{2}}
\|\partial^\alpha u^\varepsilon\|_{L^2}\nonumber \\
\lesssim&
\varepsilon\delta
\|\nabla u^\varepsilon\|_{H^3}^2 .\label{G3.15}
\end{align}
%Similarly, 
%\begin{align*}
%J_{14}^{2}
%=&\,-\varepsilon(\mu+\lambda)
%\int_{\mathbb R^3}
%\frac{q^\varepsilon}{1+\varepsilon q^\varepsilon}
%|\operatorname{div}\partial^\alpha u^\varepsilon|^2{\rm d}x
%-\varepsilon(\mu+\lambda)
%\int_{\mathbb R^3}
%\nabla\Big(
%\frac{q^\varepsilon}{1+\varepsilon q^\varepsilon}
%\Big)\cdot
%\partial^\alpha u^\varepsilon\,
%\operatorname{div}\partial^\alpha u^\varepsilon{\rm d}x
%\\
%&+\varepsilon(\mu+\lambda)
%\int_{\mathbb R^3}
%\Big[
%\partial^\alpha,
%\frac{q^\varepsilon}{1+\varepsilon q^\varepsilon}
%\nabla\operatorname{div}\Big]u^\varepsilon
%\cdot\partial^\alpha u^\varepsilon{\rm d}x .
%\end{align*}
Similar to \eqref{G3.15}, we deduce that
\begin{align}\label{G3.16}
|J_{6.2}|
\lesssim
\varepsilon\delta
\|\nabla u^\varepsilon\|_{H^3}^2.     
\end{align}
Combining \eqref{G3.15} and \eqref{G3.16} yields
\begin{align}\label{G3.17}
|J_{6}|
\lesssim \delta
\|\nabla u^\varepsilon\|_{H^3}^2 .    
\end{align}

Substituting \eqref{G3.11}--\eqref{G3.13} and \eqref{G3.17} into \eqref{G3.9}, we obtain the desired \eqref{G3.7}. The proof of Lemma \ref{L3.2} is complete.
\end{proof}

Finally, we recover the dissipation of the density fluctuations \(q^\varepsilon\) and \(r^\varepsilon\) via the interactive acoustic energy.
\begin{lem}\label{L3.3}
Let $(q^\varepsilon, u^\varepsilon, r^\varepsilon,\omega^\varepsilon)$ be the strong solution to the Cauchy problem \eqref{I3}--\eqref{I3-1}. There exists a positive constant $\eta_3$, independent of $\varepsilon$, such that
\begin{align}\label{G3.18}
&\varepsilon\frac{\rm d}{{\rm d}t}
\sum_{|\alpha|\leq2}
\int_{\mathbb R^3}
\big(
\partial^\alpha u^\varepsilon\cdot\nabla\partial^\alpha q^\varepsilon
+
\partial^\alpha\omega^\varepsilon\cdot\nabla\partial^\alpha r^\varepsilon
\big) {\rm d}x\nonumber\\
& \quad 
+\eta_3\|\nabla(q^\varepsilon,r^\varepsilon)\|_{H^2}^2
\lesssim
\|\nabla u^\varepsilon\|_{H^3}^2
+
\|u^\varepsilon-\omega^\varepsilon\|_{H^3}^2.
\end{align}
\end{lem}

\begin{proof}
We first recover the dissipation of \(q^\varepsilon\). Applying $\partial^\alpha$ with $|\alpha|\leq 2$ to \eqref{I3}$_2$, we have
\begin{align}\label{G3.19}
P'(1)\|\nabla\partial^\alpha q^\varepsilon\|_{L^2}^2
=&
-\varepsilon
\int_{\mathbb R^3}
\nabla\partial^\alpha q^\varepsilon\cdot
\partial^\alpha\partial_t u^\varepsilon {\rm d}x
-\varepsilon
\int_{\mathbb R^3}
\nabla\partial^\alpha q^\varepsilon\cdot
\partial^\alpha(u^\varepsilon\cdot\nabla u^\varepsilon) {\rm d}x
\nonumber\\
&+\varepsilon
\int_{\mathbb R^3}
\nabla\partial^\alpha q^\varepsilon\cdot
\partial^\alpha\Big(
\frac{\mu}{1+\varepsilon q^\varepsilon}\Delta u^\varepsilon
+
\frac{\mu+\lambda}{1+\varepsilon q^\varepsilon}
\nabla\operatorname{div}u^\varepsilon
\Big) {\rm d}x
\nonumber\\
&-
\int_{\mathbb R^3}
\nabla\partial^\alpha q^\varepsilon\cdot
\partial^\alpha\bigg(
\Big(
\frac{P'(1+\varepsilon q^\varepsilon)}
     {1+\varepsilon q^\varepsilon}
-P'(1)
\Big)\nabla q^\varepsilon
\bigg) {\rm d}x\nonumber\\
&+\varepsilon
\int_{\mathbb R^3}
\nabla\partial^\alpha q^\varepsilon\cdot
\partial^\alpha\Big(
\frac{1+\varepsilon r^\varepsilon}
     {1+\varepsilon q^\varepsilon}
(\omega^\varepsilon-u^\varepsilon)
\Big) {\rm d}x
\equiv:\sum_{j=1}^{5}K_j.
\end{align}
For the term \(K_1\), we use the mass equation \eqref{I3}$_1$ to obtain
\begin{align}\label{G3.20}
K_1
=&-\varepsilon\frac{\rm d}{{\rm d}t}
\int_{\mathbb R^3}
\nabla\partial^\alpha q^\varepsilon\cdot
\partial^\alpha u^\varepsilon {\rm d}x
+
\varepsilon
\int_{\mathbb R^3}
\nabla\partial^\alpha\partial_t q^\varepsilon\cdot
\partial^\alpha u^\varepsilon {\rm d}x
\nonumber\\
=&-\varepsilon\frac{\rm d}{{\rm d}t}
\int_{\mathbb R^3}
\nabla\partial^\alpha q^\varepsilon\cdot
\partial^\alpha u^\varepsilon{\rm d}x+
\int_{\mathbb R^3}
\partial^\alpha\Big(
(1+\varepsilon q^\varepsilon)\operatorname{div}u^\varepsilon
+
\varepsilon u^\varepsilon\cdot\nabla q^\varepsilon
\Big)
\operatorname{div}\partial^\alpha u^\varepsilon\,{\rm d}x
\nonumber\\
\leq&
-\varepsilon\frac{\rm d}{{\rm d}t}
\int_{\mathbb R^3}
\nabla\partial^\alpha q^\varepsilon\cdot
\partial^\alpha u^\varepsilon\,{\rm d}x
+
C\|\nabla u^\varepsilon\|_{H^2}^2 .
\end{align}
For the remaining terms $K_2,\dots, K_5$, we get by a direct calculation that
\begin{align}\label{G3.21}
K_2
\leq&\,
\frac{1}{5}\|\nabla\partial^\alpha q^\varepsilon\|_{L^2}^2
+
C\|u^\varepsilon\cdot\nabla u^\varepsilon\|_{H^2}^2\leq
\frac{1}{5}\|\nabla\partial^\alpha q^\varepsilon\|_{L^2}^2
+
C\delta^2\|\nabla u^\varepsilon\|_{H^2}^2 ,
\\
\label{G3.22}
K_3
\leq&\,
\frac{1}{5}\|\nabla\partial^\alpha q^\varepsilon\|_{L^2}^2
+
C\varepsilon^2\|\nabla u^\varepsilon\|_{H^3}^2 \big(1+\|\nabla q^\varepsilon\|_{H^2}^2\big)  \leq
\frac{1}{5}\|\nabla\partial^\alpha q^\varepsilon\|_{L^2}^2
+
C\varepsilon^2 \|\nabla u^\varepsilon\|_{H^3}^2  ,
\\
\label{G3.24}
K_4
\leq&\,
C\varepsilon\|q^\varepsilon\|_{H^3}
\|\nabla q^\varepsilon\|_{H^2}
\|\nabla\partial^\alpha q^\varepsilon\|_{L^2}
\leq
C\delta\|\nabla q^\varepsilon\|_{H^2}^2,
\\
\label{G3.23}
K_5
\leq&\,
\frac{1}{5}\|\nabla\partial^\alpha q^\varepsilon\|_{L^2}^2
+
C\varepsilon^2\|(q^\varepsilon,r^\varepsilon)\|_{H^3}^2 \|u^\varepsilon-\omega^\varepsilon\|_{H^2}^2\leq
\frac{1}{5}\|\nabla\partial^\alpha q^\varepsilon\|_{L^2}^2
+
C\varepsilon^2 \delta^2 \|u^\varepsilon-\omega^\varepsilon\|_{H^2}^2.
\end{align}
Substituting the estimates \eqref{G3.20}--\eqref{G3.23} into \eqref{G3.19} and summing over $|\alpha|\leq 2$, we deduce
\begin{align}\label{G3.25}
&\varepsilon\frac{\rm d}{{\rm d}t}
\sum_{|\alpha|\leq2}
\int_{\mathbb R^3}
\partial^\alpha u^\varepsilon\cdot
\nabla\partial^\alpha q^\varepsilon {\rm d}x
+
\eta_{4}\|\nabla q^\varepsilon\|_{H^2}^2
 \lesssim
\|\nabla u^\varepsilon\|_{H^3}^2
+
\|u^\varepsilon-\omega^\varepsilon\|_{H^2}^2,
\end{align}
for some constant $\eta_{4}>0$. 

Similarly, for the dissipation of \(r^\varepsilon\), we use \eqref{I3}$_4$ to obtain 
\begin{align}\label{G3.31}
&\varepsilon\frac{\rm d}{{\rm d}t}
\sum_{|\alpha|\leq2}
\int_{\mathbb R^3}
\partial^\alpha\omega^\varepsilon\cdot
\nabla\partial^\alpha r^\varepsilon {\rm d}x
+
\eta_{5}\|\nabla r^\varepsilon\|_{H^2}^2\lesssim
\|\nabla u^\varepsilon\|_{H^3}^2
+
\|u^\varepsilon-\omega^\varepsilon\|_{H^3}^2,
\end{align}
for some constant $\eta_{5}>0$.

Putting \eqref{G3.25} and \eqref{G3.31} together yields \eqref{G3.18}.  The proof of Lemma \ref{L3.3} is complete.
\end{proof}

\subsection{Proof of Theorem \ref{Th1}}
Now we prove the
existence, uniqueness and uniform-in-$\varepsilon$ estimates of global strong solution $(q^\varepsilon,u^\varepsilon,r^\varepsilon,\omega^\varepsilon)$ to the scaled compressible Navier--Stokes--Euler system \eqref{I3}--\eqref{I3-1}.
%\begin{proof}%[Proof of Theorem \ref{Th1}]

We define the energy functional $\mathcal{X}_{\varepsilon}(t)$ by
\begin{align}\label{G3.32}
\mathcal X_\varepsilon(t)
:=&\,P'(1)\|q^\varepsilon(t)\|_{L^2}^2
+\|(u^\varepsilon,r^\varepsilon,\omega^\varepsilon)(t)\|_{L^2}^2
\nonumber\\
&+\sum_{1\leq|\alpha|\leq3}
\bigg(
\bigg\|
\sqrt{\frac{P'(1+\varepsilon q^\varepsilon)}
     {1+\varepsilon q^\varepsilon}}
\partial^\alpha q^\varepsilon
\bigg\|_{L^2}^2
+
\Big\|
\frac{\partial^\alpha r^\varepsilon}
     {\sqrt{1+\varepsilon r^\varepsilon}}
\Big\|_{L^2}^2
+
\|\partial^\alpha(u^\varepsilon,\omega^\varepsilon)\|_{L^2}^2
\bigg)
\nonumber\\
&+\kappa\varepsilon
\sum_{|\alpha|\leq2}
\int_{\mathbb R^3}
 (
\partial^\alpha u^\varepsilon\cdot\nabla\partial^\alpha q^\varepsilon
+
\partial^\alpha\omega^\varepsilon\cdot\nabla\partial^\alpha r^\varepsilon
 ){\rm d}x ,
\end{align}
where \(0<\kappa\ll1\) will be chosen later. The corresponding dissipation
functional is given by
\begin{align}\label{G3.33}
\mathcal D_\varepsilon(t)
:=&
\|\nabla u^\varepsilon(t)\|_{H^3}^2
+\|(u^\varepsilon-\omega^\varepsilon)(t)\|_{H^3}^2
+\|\nabla(q^\varepsilon,r^\varepsilon)(t)\|_{H^2}^2
+\|\nabla\omega^\varepsilon(t)\|_{H^2}^2 .
\end{align}
Owing to Cauchy’s inequality, it holds that
\begin{align*}
\bigg|
\varepsilon
\sum_{|\alpha|\leq2}
\int_{\mathbb R^3}
 (
\partial^\alpha u^\varepsilon\cdot\nabla\partial^\alpha q^\varepsilon
+
\partial^\alpha\omega^\varepsilon\cdot\nabla\partial^\alpha r^\varepsilon
 ) {\rm d}x
\bigg|
\lesssim
\varepsilon\|(q^\varepsilon,u^\varepsilon,r^\varepsilon,\omega^\varepsilon)\|_{H^3}^2 .    
\end{align*}
Thus,  choosing $\varepsilon\in (0,1)$ sufficiently small, we deduce that
\begin{align}\label{G3.34}
\mathcal X_\varepsilon(t)
\sim
\|(q^\varepsilon,u^\varepsilon,r^\varepsilon,\omega^\varepsilon)(t)\|_{H^3}^2.
\end{align}

Adding 
\eqref{G3.2}, \eqref{G3.7} and \(\kappa\times\eqref{G3.18}\) together, and choosing $\kappa\in (0,1)$ sufficiently small, we deduce the following differential inequality:
\begin{align}\label{G3.35}
\frac{{\rm d}}{{\rm d}t}\mathcal{X}_{\varepsilon}(t)+  \eta_{6}   \mathcal{D}_{\varepsilon}(t)\leq 0,
\end{align}
for some constant \(\eta_6>0\),
where we have used  the inequality
\begin{align*}
\|\nabla\omega^\varepsilon\|_{H^2}^2
\lesssim
\|\nabla u^\varepsilon\|_{H^3}^2
+\|u^\varepsilon-\omega^\varepsilon\|_{H^3}^2 .
\end{align*}

Integrating \eqref{G3.35} over \([0,t]\), we obtain
\begin{align}\label{G3.36}
\mathcal X_\varepsilon(t)
+\eta_{6}\int_0^t\mathcal D_\varepsilon(\tau)\,d\tau
\leq
\mathcal X_\varepsilon(0),
\qquad 0\leq t<T .
\end{align}
Using the equivalence \eqref{G3.34}, we further obtain
\begin{align}\label{G3.37}
&\sup_{0\leq t<T}
\|(q^\varepsilon,u^\varepsilon,r^\varepsilon,\omega^\varepsilon)(t)\|_{H^3}^2
+\int_0^T
\big(
\|(u^\varepsilon-\omega^\varepsilon)(\tau)\|_{H^3}^2
+\|\nabla(q^\varepsilon,r^\varepsilon,\omega^\varepsilon)(\tau)\|_{H^2}^2
+\|\nabla u^\varepsilon(\tau)\|_{H^3}^2
\big){\rm d}\tau
\nonumber\\
&\qquad
\leq C
\|(q_0^\varepsilon,u_0^\varepsilon,r_0^\varepsilon,\omega_0^\varepsilon)\|_{H^3}^2 ,
\end{align}
where \(C > 0\) is a constant independent of \(T\) and \(\varepsilon\).

%We now complete the proof of Theorem \ref{Th1}. 
For the local well-posedness, using the same approach as \cite{Matsumura-Nishida-1979,Matsumura-Nishida-1980}, we can construct the unique local strong solution to the Cauchy problem \eqref{I3}--\eqref{I3-1}. Such a solution can be extended to a global one by a continuity argument and the a priori estimate \eqref{G3.37}. Moreover, the solution satisfies the bound \eqref{TG2}.  The proof of Theorem \ref{Th1} is complete.
\hfill $\square$ %\end{proof}

\section{Analysis of the limiting incompressible system}
This section is devoted to the global existence, uniqueness, and time-decay rates of strong solutions to the limiting incompressible two-phase system 
\eqref{I4}--\eqref{I4-1}. Local well-posedness follows from the standard iteration scheme for hyperbolic--parabolic fluid systems; see \cite{ZWXM-ZAMP-2021}. We therefore focus on the uniform estimates \eqref{TGG2} and time decay rates \eqref{TGG3}-\eqref{TGG4}.

\subsection{A priori estimates}
Let \((u,\omega)\) be a strong solution to the system \eqref{I4}--\eqref{I4-1}
on \([0,T_1)\), with an arbitrary constant \(T_1>0\). Assume that 
\begin{align}\label{G4.1}
\sup_{0\leq t<T_1}
\|(u,\omega)(t)\|_{H^3}
\leq \sigma,
\end{align}
where \(0<\sigma<1\) is chosen sufficiently small.

We begin with the zero-order estimate for \((u,\omega)\).

\begin{lem}\label{L4.1}
Let $( u,\omega)$ be the strong solution to the incompressible two-phase system \eqref{I4}--\eqref{I4-1}. Then the following differential equality holds:
\begin{align}
\label{G4.2}
 \frac{{\rm d}}{{\rm d} t} \|(u,\omega)\|_{L^2}^2 
+2\mu\|\nabla u\|_{L^2}^2+
2\|u-\omega\|_{L^2}^2=0.
\end{align}
\end{lem}

\begin{proof}
Multiplying \eqref{I4}$_1$ by $u$ and  \eqref{I4}$_3$ by $\omega$, integrating over $\mathbb R^3$, and then applying the divergence-free conditions \eqref{I4}$_2$ and \eqref{I4}$_4$, we have
\begin{align*}
&\frac{1}{2}\frac{{\rm d}}{{\rm d}t}\|(u,\omega)\|_{L^2}^2+\mu \|\nabla u\|_{L^2}^2+\|u-\omega\|_{L^2}^2\nonumber\\
=&\,-\int_{\mathbb R^3} u\cdot (u\cdot\nabla u){\rm d} x-\int_{\mathbb R^3} \omega\cdot (\omega\cdot\nabla \omega){\rm d} x-\int_{\mathbb R^3}\nabla\pi_{1}\cdot u{\rm d}x-\int_{\mathbb R^3}\nabla {\pi}_{2}\cdot\omega{\rm d}x\nonumber\\
=&\, \frac{1}{2}\int_{\mathbb R^3} {\rm div}\,u|u|^2{\rm d}x+ \frac{1}{2}\int_{\mathbb R^3} {\rm div}\,\omega|\omega|^2{\rm d}x+\int_{\mathbb R^3} \pi_{1}{\rm div}\,u{\rm d}x+\int_{\mathbb R^3} \pi_{2}{\rm div}\,\omega{\rm d}x \nonumber\\
=&\,0,
\end{align*}
which yields \eqref{G4.2}.
\end{proof}

Next, we derive the higher-order energy estimates for \((u,\omega)\). 

\begin{lem}\label{L4.2}
There exists a positive constant $\eta_1^\prime$, such that
\begin{align}\label{G4.3}
 \frac{{\rm d}}{{\rm d}t} \sum_{1\leq |\alpha|\leq 3}\|(\partial^\alpha u,\partial^\alpha\omega)\|_{L^2}^2+\eta_1^\prime\sum_{1\leq\alpha\leq 3}\big(\|\partial^\alpha u-\partial^\alpha\omega\|_{L^2}^2+\|\nabla \partial^\alpha u\|_{L^2}^2 \big) \leq 0.
\end{align}
\end{lem}

\begin{proof}
Applying $\partial^\alpha$ with $1\leq|\alpha|\leq3$ to the equations \eqref{I4}$_1$ and \eqref{I4}$_3$ and taking inner product of the resulting equations with $\partial^\alpha u$ and $\partial^\alpha \omega$, respectively, we obtain
\begin{align}\label{G4.4}
&\frac{1}{2}\frac{{\rm d}}{{\rm d}t}\|(\partial^\alpha u,\partial^\alpha \omega
)\|_{L^2}^2+   \|\partial^\alpha(u-\omega)\|_{L^2}^2+\mu \|\nabla\partial^\alpha u\|_{L^2}^2 \nonumber\\
=&\,-\int_{\mathbb R^3}[\partial^\alpha,u\cdot\nabla] u\cdot\partial^\alpha u{\rm d}x-\int_{\mathbb R^3}[\partial^\alpha,\omega\cdot\nabla] \omega\cdot\partial^\alpha \omega{\rm d}x+\frac{1}{2}\int_{\mathbb R^3}|\partial^\alpha u|^2{\rm div}\,u{\rm d}x\nonumber\\
&+\frac{1}{2}\int_{\mathbb R^3}|\partial^\alpha \omega|^2{\rm div}\,\omega{\rm d}x+\int_{\mathbb R^3} \partial^\alpha \pi_1\partial^\alpha{\rm div}\,u{\rm d}x+\int_{\mathbb R^3} \partial^\alpha \pi_2\partial^\alpha{\rm div}\,\omega{\rm d}x\nonumber\\
=&\,-\int_{\mathbb R^3}[\partial^\alpha,u\cdot\nabla] u\cdot\partial^\alpha u{\rm d}x-\int_{\mathbb R^3}[\partial^\alpha,\omega\cdot\nabla] \omega\cdot\partial^\alpha \omega{\rm d}x\nonumber\\
\equiv:&\,I_{1}^\prime+I_{2}^\prime.
\end{align}
Thanks to  Lemma \ref{L2.4}, we have
\begin{align}\label{G4.5}
 I_{1}^\prime+I_{2}^\prime\lesssim&\, (\|[\partial^\alpha,u\cdot\nabla]u\|_{L^2}+\|[\partial^\alpha,\omega\cdot\nabla]\omega\|_{L^2}) \|\partial^\alpha(u,\omega)\|_{L^2}\nonumber\\
 \lesssim&\, \|(u,\omega)\|_{H^3}\|\nabla(u,\omega)\|_{H^2}^2\nonumber\\
 \lesssim&\,\sigma \|\nabla u\|_{H^2}^2+\sigma\|\nabla(u-\omega)\|_{H^2}^2.
\end{align}
Substituting \eqref{G4.5} into \eqref{G4.4} and summing over $1\leq|\alpha|\leq 3$, we obtain \eqref{G4.3}.
\end{proof}

\begin{proof}[Proof of Theorem \ref{Th2}: global existence and uniqueness of classical solutions]
We now prove the global well-\linebreak posedness of  Cauchy problem \eqref{I4}--\eqref{I4-1}. It follows from \eqref{G4.2} and \eqref{G4.3} that 
\begin{align}
\frac{{\rm d}}{{\rm d}t}\|(u,\omega)\|_{H^3}^2+ \eta^\prime_{2}\big ( \|\nabla u\|_{H^3}^2+  \|u-\omega\|_{H^3}^2   \big) \leq 0,\label{gron}
\end{align}
for some constant $\eta_{2}^\prime>0$. Integrating  the above inequality over \([0,t]\) yields
\begin{align}
\label{G4.6}
&\|(u,\omega)(t)\|_{  H^3}^2
+\eta_2^\prime
\int_0^t
\big(
\|(u-\omega)(\tau)\|_{  H^3}^2
+\|\nabla u(\tau)\|_{  H^3}^2
\big) {\rm d}\tau\leq
\|(u_0,\omega_0)\|_{ H^3}^2,
\end{align}
for any $0\leq t<T_1$. 
Moreover, by choosing $\|(u_0,\omega_0)\|_{H^3}$ sufficiently small, the bootstrap assumption \eqref{G4.1} is justified. Therefore, the local strong solution extends globally by the above {\it uniform a priori estimates} of $(u,\omega)$ and the standard continuity argument. In particular, the inequality \eqref{G4.6} holds for all $t \geq 0$, which yields \eqref{TGG2}.
\end{proof}

\subsection{Linear decay estimate}
In this subsection, we derive the decay estimates of the following system
\begin{equation}\label{G4.7}
\left\{
\begin{aligned}
&\partial_tu-\mu\Delta u+( u-\omega)=0,\\
&\operatorname{div}u=0,\\
&\partial_t\omega -(u-\omega)=0 ,\\
&\operatorname{div}\omega=0,
\end{aligned}
\right.
\end{equation}
with the initial data
\begin{align}\label{G4.8}
 (u , \omega )|_{t=0}= ( u_0, \omega _0),
\end{align}
where $\operatorname{div}u_0=\operatorname{div}\omega_0=0$. This is the linearized system of  \eqref{I4}--\eqref{I4-1} in the space $L^2_\sigma(\mathbb{R}^3)$. For brevity, we write $U(t):=(u(t),\omega(t))$ and $U_0:=(u_0,\omega_0)$ and represent the solution of \eqref{G4.7}--\eqref{G4.8} as
\begin{align*}
U(t)=\mathbb{A}(t) U_0.
\end{align*}
Here $\mathbb{A}(t)$ is the solution operator of the system \eqref{G4.7}--\eqref{G4.8}. Now we establish the $L^p$--$L^q$ time-decay property of the solution operator $\mathbb{A}(t)$.
\begin{prop}\label{T4.1}
Let \(1\leq q\leq2\). For any multi-indices \(\alpha,\alpha'\) with
\(\alpha'\leq\alpha\) and \(m=|\alpha-\alpha'|\), it holds that 
\begin{align}\label{G4.9}
 \|\partial^{\alpha}\mathbb{A}(t)U_{0}\|_{L^{2}}
\lesssim  &\, (1+t)^{-\frac{3}{2} (\frac{1}{q}-\frac{1}{2} )-\frac{m}{2}}
\big(\|\partial^{\alpha^{\prime}}U_{0}\|_{L^{q}}
+\|\partial^{\alpha}U_{0}\|_{L^{2}}\big).
\end{align}
Moreover, the relaxation term $u-\omega$ exhibits a faster decay
\begin{align}\label{G4.10}
 \|\partial^{\alpha}u(t)-\partial^\alpha\omega(t)\|_{L^{2}}
\lesssim  &\, (1+t)^{-\frac{3}{2} (\frac{1}{q}-\frac{1}{2} )-\frac{m+1}{2}}
\big(\|\partial^{\alpha^{\prime}}U_{0}\|_{L^{q}}
+\|\partial^{\alpha}U_{0}\|_{L^{2}}\big).
\end{align}
\end{prop}

\begin{proof}
Taking the Fourier transform with respect to \(x\) for the equations \eqref{G4.7}$_1$ and \eqref{G4.7}$_3$, and utilizing the equations \eqref{G4.7}$_2$ and \eqref{G4.7}$_4$, we obtain 
\begin{equation}\label{G4.11}
\left\{
\begin{aligned}
&\partial_t\widehat u+(\mu|\xi|^2+1)\widehat u-\widehat\omega=0,\\
&\partial_t\widehat\omega-\widehat u+\widehat\omega=0.
\end{aligned}
\right.
\end{equation}
Then, \eqref{G4.11} can be rewritten into the following matrix form:
\begin{align*}
\partial_t
\begin{pmatrix}
\widehat u\\
\widehat\omega
\end{pmatrix}
=
A(\xi)
\begin{pmatrix}
\widehat u\\
\widehat\omega
\end{pmatrix},
\end{align*}
where $A(\xi)$ is given by
\begin{align*}
 A(\xi):=
\begin{pmatrix}
-(\mu|\xi|^2+1)&1\\
1&-1
\end{pmatrix}.   
\end{align*}
By a direct calculation, we obtain the following eigenvalues of \(A(\xi)\):
\begin{align*}
\lambda_\pm(\xi)
=
-\frac{\mu|\xi|^2+2}{2}
\pm
\frac12\sqrt{\mu^2|\xi|^4+4}.
\end{align*}

On the one hand, in the low frequencies regime \(|\xi|\leq1\), we have
\begin{align*}
\lambda_+(\xi)
=
-\frac{\mu}{2}|\xi|^2+\mathcal{O}(|\xi|^4),
\quad
\lambda_-(\xi)
=
-2+\mathcal{O}(|\xi|^2).    
\end{align*}
On the other hand, in the high frequencies \(|\xi|\geq1\), since \(\mu>0\) is fixed, there exists a constant \(c>0\), which depends only on \(\mu\), such that
\begin{align*}
\lambda_+(\xi)\leq -c,
\quad
\lambda_-(\xi)\leq -c(1+|\xi|^2).    
\end{align*}
Therefore, for all \(\xi\in\mathbb R^3\), it holds that
\begin{equation}
\label{G4.12}
\lambda_+(\xi)
\leq
-c\frac{|\xi|^2}{1+|\xi|^2},
\quad
\lambda_-(\xi)\leq -c.
\end{equation}

Let \(P_\pm(\xi)\) be the spectral projectors associated with
\(\lambda_\pm(\xi)\). Since
\begin{align*}
\lambda_+(\xi)-\lambda_-(\xi)
=
\sqrt{\mu^2|\xi|^4+4}
\geq 2,    
\end{align*}
the projectors $P_\pm(\xi)$ are therefore uniformly bounded for all frequencies
\(\xi\in\mathbb R^3\). Thus, we further have
\begin{align*}
e^{tA(\xi)}
=
e^{\lambda_+(\xi)t}P_+(\xi)
+
e^{\lambda_-(\xi)t}P_-(\xi),    
\end{align*} 
which, together with \eqref{G4.12}, yields
\begin{align*}
\big|\widehat{\mathbb A(t)U_0}(\xi)\big|
\lesssim\exp\Big(
-c\frac{|\xi|^2}{1+|\xi|^2}t
\Big)
|\widehat U_0(\xi)|.
\end{align*}
In particular, in the low-frequency regime \(|\xi|\leq1\), we have
\begin{align}
\label{G4.low-U}
\big|\widehat{\mathbb A(t)U_0}(\xi)\big|
\lesssim
e^{-c|\xi|^2t}|\widehat U_0(\xi)|,
\end{align}
whereas in the high-frequency region \(|\xi|\geq1\), we have
\begin{align}
\label{G4.high-U}
\big|\widehat{\mathbb A(t)U_0}(\xi)\big|
\lesssim e^{-ct}|\widehat U_0(\xi)|.
\end{align}

We next estimate the relaxation term $u-\omega$.
For each fixed \(\xi\), we can compute an eigenvector associated with \(\lambda_\pm(\xi)\)
as follows
\begin{align*}
r_\pm(\xi)=
\begin{pmatrix}
1\\
1+\mu|\xi|^2+\lambda_\pm(\xi)
\end{pmatrix}.    
\end{align*} 
Since
\begin{align*}
|\det(r_+(\xi),r_-(\xi))|
=
|\lambda_+(\xi)-\lambda_-(\xi)|
=
\sqrt{\mu^2|\xi|^4+4}\geq2,    
\end{align*}
the decomposition with respect to \(r_+(\xi)\) and \(r_-(\xi)\) is uniformly
bounded. Thus, it holds that
\begin{align*}
\widehat{\mathbb A(t)U_0}(\xi)
=
c_+(\xi)e^{\lambda_+(\xi)t}r_+(\xi)
+
c_-(\xi)e^{\lambda_-(\xi)t}r_-(\xi),    
\end{align*}
with
\begin{align*}
 |c_+(\xi)|+|c_-(\xi)|
\leq C|\widehat U_0(\xi)|.   
\end{align*}

Now we estimate the relaxation term $\hat{u}(t,\xi)-\hat{\omega}(t,\xi)$. Taking the difference of the two components gives
\begin{align*}
\hat{u}(t,\xi)-\hat{\omega}(t,\xi)
=
-c_+(\xi)e^{\lambda_+(\xi)t}
\big(\mu|\xi|^2+\lambda_+(\xi)\big)
-c_-(\xi)e^{\lambda_-(\xi)t}
\big(\mu|\xi|^2+\lambda_-(\xi)\big).    
\end{align*}

For \(|\xi|\leq1\), using
\begin{align*}
\lambda_+(\xi)
=
-\frac{\mu}{2}|\xi|^2+\mathcal{O}(|\xi|^4),
\quad
\lambda_-(\xi)=-2+\mathcal{O}(|\xi|^2),    
\end{align*}
we get
\begin{align*}
|\mu|\xi|^2+\lambda_+(\xi)|\lesssim |\xi|^2\lesssim |\xi|,
\quad
|\mu|\xi|^2+\lambda_-(\xi)|\lesssim 1,   
\end{align*}
which, together with
\begin{align*}
e^{\lambda_+(\xi)t}\leq e^{-c|\xi|^2t},
\quad
e^{\lambda_-(\xi)t}\leq e^{-ct},    
\end{align*}
yields
\begin{equation}
\label{G4.low-d}
\big|\hat{u}(t,\xi)-\hat{\omega}(t,\xi)\big|
\leq
C\big(|\xi|e^{-c|\xi|^2t}+e^{-ct}\big)
|\widehat U_0(\xi)|,\quad \text{for}
\quad |\xi|\leq1.
\end{equation}
For \(|\xi|\geq1\), both eigenmodes are exponentially damped. Hence, we obtain
\begin{equation}
\label{G4.high-d}
\big|\hat{u}(t,\xi)-\hat{\omega}(t,\xi)\big|
\leq
Ce^{-ct}|\widehat U_0(\xi)|.
\end{equation}

We now prove the decay estimates \eqref{G4.9}. We split the frequency space $\mathbb{R}^3_\xi$ into the low-frequency part \(|\xi|\leq1\) and the high-frequency part  {\(|\xi|>1\)}. Then, it  follows from \eqref{G4.low-U} that
\begin{align*}
\|\partial^\alpha\mathbb A(t)U_0\|_{L^2(|\xi|\leq1)}
\lesssim
\big\|
|\xi|^m e^{-c|\xi|^2t}
\widehat{\partial^{\alpha'}U_0}
\big\|_{L^2_\xi}.    
\end{align*}
Let \(q'\) be the conjugate exponent of \(q\), and choose \(r\) such that
\begin{align*}
\frac1r+\frac1{q'}=\frac12.   
\end{align*}
Applying H\"{o}lder’s inequality and the Hausdorff--Young inequality, we obtain 
\begin{align}\label{G4.17}
 \big\||\xi|^m e^{-c|\xi|^2t}
\widehat{\partial^{\alpha'}U_0}
 \big\|_{L^2_\xi}
 \leq
\big\||\xi|^m e^{-c|\xi|^2t}\big\|_{L^r_\xi}
\big\|\widehat{\partial^{\alpha'}U_0}\big\|_{L^{q'}_\xi}
 \leq
C(1+t)^{-\frac32(\frac1q-\frac12)-\frac m2}
\|\partial^{\alpha'}U_0\|_{L^q}.   
\end{align}
For the high-frequency part, \eqref{G4.high-U} implies that 
\begin{align}\label{G4.18}
\|\partial^\alpha\mathbb A(t)U_0\|_{L^2( {|\xi|>1})}
\leq
Ce^{-ct}\|\partial^\alpha U_0\|_{L^2}
\leq
C(1+t)^{-\frac32(\frac1q-\frac12)-\frac m2}
\|\partial^\alpha U_0\|_{L^2}.    
\end{align}
Combining the estimates \eqref{G4.17} and \eqref{G4.18} gives \eqref{G4.9}.

It remains to prove \eqref{G4.10}. From \eqref{G4.low-d}, we get
\begin{align*}
\|\partial^\alpha u(t)-\partial^\alpha\omega(t)\|_{L^2(|\xi|\leq1)}
&\lesssim
\big\|
|\xi|^{m+1}e^{-c|\xi|^2t}
\widehat{\partial^{\alpha'}U_0}
\big\|_{L^2_\xi}
+
e^{-ct}
\big\|
|\xi|^m\widehat{\partial^{\alpha'}U_0}
\big\|_{L^2(|\xi|\leq1)}.    
\end{align*}
By a direct calculation, we can bound the first term as
\begin{align*}
\big\||\xi|^{m+1}e^{-c|\xi|^2t}
\widehat{\partial^{\alpha'}U_0}
\big \|_{L^2_\xi}
\lesssim (1+t)^{-\frac32(\frac1q-\frac12)-\frac{m+1}{2}}
\|\partial^{\alpha'}U_0\|_{L^q}.    
\end{align*}
For the second term, it holds that 
\begin{align*}
\big\|
|\xi|^m\widehat{\partial^{\alpha'}U_0}
\big\|_{L^2(|\xi|\leq1)}
\lesssim\|\partial^{\alpha'}U_0\|_{L^q}.    
\end{align*}
Consequently, we obtain 
\begin{align}\label{G4.19}
\|\partial^\alpha u(t)-\partial^\alpha\omega(t)\|_{L^2(|\xi|\leq1)} 
\lesssim (1+t)^{-\frac32(\frac1q-\frac12)-\frac{m+1}{2}}
\|\partial^{\alpha'}U_0\|_{L^q}.       
\end{align}
For the high-frequency part \( {|\xi|>1}\), we use \eqref{G4.high-d} to obtain
\begin{align}\label{G4.20}
\|\partial^\alpha u(t)-\partial^\alpha\omega(t)\|_{L^2({|\xi|>1})}
\lesssim
 e^{-ct}\|\partial^\alpha U_0\|_{L^2}.    
\end{align}
Based on the estimates \eqref{G4.19} and \eqref{G4.20}, we obtain \eqref{G4.10}. Thus, the proof of Proposition \ref{T4.1} is complete.
\end{proof}

\subsection{Nonlinear analysis}
We now prove the time-decay rates of the nonlinear system
\eqref{I4}--\eqref{I4-1}. By Duhamel principle, we first rewrite \eqref{I4} in the following form:
\begin{equation}\label{G4.21}
 U(t)=\mathbb A(t)U_0+
\int_0^t \mathbb A(t-\tau)\big(N_1(\tau),N_2(\tau)\big) {\rm d}\tau,
\end{equation}
where
\begin{align*}
U=(u,\omega),\qquad U_0=(u_0,\omega_0),    
\end{align*}
and
\begin{equation}\label{G4.22}
N_1:=-\mathbb P{\rm div}(u\otimes u),
\qquad
N_2:=-\mathbb P{\rm div}(\omega\otimes\omega).
\end{equation}
%Taking the divergence of  \eqref{I4}$_1$ and \eqref{I4}$_3$,
%we obtain
%\begin{align}\label{pi}
%-\Delta\pi_1=\partial_i\partial_j(u_i u_j),
%\qquad
%-\Delta\pi_2=\partial_i\partial_j(\omega_i\omega_j).
%\end{align}
%Hence, in Fourier variables, we have
%\begin{align*}
%\widehat{\partial_\ell\pi_1}(\xi)
%=
%-i\xi_\ell\frac{\xi_i\xi_j}{|\xi|^2}
%\widehat{u_i u_j}(\xi),    
%\end{align*}
%which results in
%\begin{align*}
% \widehat{N}_{1,\ell}(\xi)
%=
%-i\xi_j\widehat{u_j u_\ell}(\xi)
%+
%i\xi_\ell\frac{\xi_i\xi_j}{|\xi|^2}
%\widehat{u_i u_j}(\xi).   
%\end{align*}
Since the Leray projector \(\mathbb P\) is bounded on
\(L^q(\mathbb R^3)\) for all \(1<q<\infty\), and \(\mathbb P\)
commutes with spatial derivatives, one has
\begin{align}\label{G4.23}
\|\mathbb P g\|_{L^q}
\leq C_q\|g\|_{L^q},
\qquad
\|\nabla^m\mathbb P g\|_{L^2}
\leq C\|\nabla^m g\|_{L^2},
\end{align}
for any integer \(m\geq0\), where \(C_q>0\) depends only on \(q\).

Now we control the second term on the right-hand side of \eqref{G4.21}. For this, we need take advantage of the divergence structure of the nonlinear terms. More specifically, we rewrite  \eqref{G4.22} as follows:
\begin{align}\label{G4.24}
(N_1,N_2)
=
-\sum_{j=1}^3
\partial_j
\big(
\mathbb P(u_j u),
\mathbb P(\omega_j\omega)
\big).
\end{align}
Let \(\beta\) be a multi-index with \(|\beta|=k\), where \(k=0,1\).
Since the solution operator \(\mathbb A(t)\), the Leray projector
\(\mathbb P\), and the spatial derivatives commute with each other, it follows from \eqref{G4.24} that
\begin{align*}
\partial^\beta
\mathbb A(t-\tau)(N_1,N_2)
=
-\sum_{j=1}^3
\partial^{\beta+e_j}
\mathbb A(t-\tau)
\big(
\mathbb P(u_j u),
\mathbb P(\omega_j\omega)
\big),
\end{align*}
where $e_j, j=1,2,3,$ are   unit vectors in $\mathbb{R}^3. $

Let $0<\vartheta < \frac{1}{4}$ be a constant, which is sufficiently close to $\frac{1}{4}$.  Setting \(q = \frac{3}{3 - 2\vartheta}\in(1,\frac{6}{5})\), $\alpha=\beta+e_j$ and 
$\alpha'=0$  in Proposition \ref{T4.1},   we obtain, for $k=0,1,$ that
\begin{align}\label{G4.27}
&\bigg\|
\nabla^k
\int_0^t
\mathbb A(t-\tau)\big(N_1(\tau),N_2(\tau)\big)
 {\rm d}\tau
\bigg\|_{L^2}
\nonumber\\
\leq&\, C_{\vartheta}
\int_0^t
(1+t-\tau)^{-\frac54-\frac{k}{2}+\vartheta}
\big(\|F(U(\tau))\|_{L^{\frac{3}{3 - 2\vartheta}}}+\|\nabla^{k+1}F(U(\tau))\|_{L^2}\big){\rm d}\tau\nonumber\\
\leq&\, C_\vartheta \int_0^t
(1+t-\tau)^{-\frac54-\frac{k}{2}+\vartheta}
\big(\|F(U(\tau))\|_{L^{ 1}\cap L^6}+\|\nabla^{k+1}F(U(\tau))\|_{L^2}\big){\rm d}\tau\nonumber\\
\leq&\, C_\vartheta \int_0^t
(1+t-\tau)^{-\frac54-\frac{k}{2}+\vartheta}
\big(\|F(U(\tau))\|_{L^{ 1} }+\|\nabla F(U(\tau))\|_{L^2}+\|\nabla^{k+1}F(U(\tau))\|_{L^2}\big){\rm d}\tau\nonumber\\
\leq&\, C_\vartheta \int_0^t
(1+t-\tau)^{-\frac54-\frac{k}{2}+\vartheta}\|U(\tau)\|_{H^3}^2{\rm d}\tau,
\end{align}
where we have denoted $F(U):=(u\otimes u,\omega\otimes\omega)$. Here we have used the boundedness of Leray projection \eqref{G4.23} in the first inequality, and the following inequalities
\begin{align}
\|F(U)\|_{L^1}
&\leq
C\|U\|_{L^2}^2,\nonumber\\
\|\nabla^kF(U)\|_{L^2}
&\leq
C\|U\|_{H^3}^2,~k=1,2,\nonumber
\end{align}
in the last inequality. Moreover, for the difference between $u$ and $\omega$ components, we use the same approach as \eqref{G4.27} to deduce that
\begin{align}\label{G4.28}
&\left\|
\left(
\int_0^t
\mathbb A(t-\tau)\big(N_1(\tau),N_2(\tau)\big)
\,{\rm d}\tau
\right)_u
-
\left(
\int_0^t
\mathbb A(t-\tau)\big(N_1(\tau),N_2(\tau)\big)
\,{\rm d}\tau
\right)_\omega
\right\|_{L^2}
\nonumber\\
&\leq C_{\vartheta}
\int_0^t
(1+t-\tau)^{-\frac74+\vartheta}\|U(\tau)\|_{H^3}^2{\rm d}\tau .
\end{align}

\begin{proof}[Proof of Theorem \ref{Th2}: time decay rates of classical solutions] 
We first derive $L^2$-decay of $U(t)$. Define
\begin{equation}\label{G4.33}
Y_{\infty}(t):=
\sup_{0\leq s\leq t}
\big\{
(1+s)^{\frac32}\|U(s)\|_{H^3}^2
+
(1+s)^{\frac52}\|\nabla U(s)\|_{H^2}^2
+
(1+s)^{\frac52}\|u(s)-\omega(s)\|_{L^2}^2
\big\}.
\end{equation}
From Proposition \ref{T4.1}, \eqref{G4.21} and \eqref{G4.27}, we get
\begin{align}\label{G4.34}
\|U(t)\|_{L^2}
&\lesssim
(1+t)^{-\frac34}\|U_0\|_{H^3\cap L^1}
+
\int_0^t
(1+t-\tau)^{-\frac54+\vartheta}
\|U(\tau)\|_{H^3}^2{\rm d}\tau\nonumber\\
&\lesssim (1+t)^{-\frac34}\|U_0\|_{H^3\cap L^1}+Y_\infty(t)\int_0^t(1+t-\tau)^{-\frac54+\vartheta}(1+\tau)^{-\frac32}{\rm d}\tau\nonumber\\
&\lesssim (1+t)^{-\frac34}\|U_0\|_{H^3\cap L^1}+Y_\infty(t)(1+t)^{-\frac54+\vartheta}.
\end{align}
Similarly, combining \eqref{G4.21} and \eqref{G4.27} with \(k = 1\) yields
\begin{align}\label{G4.37}
\|\nabla U(t)\|_{L^2}
&\lesssim
(1+t)^{-\frac54}\|U_0\|_{H^3\cap L^1}
+
\int_0^t
(1+t-\tau)^{-\frac74+\theta}
\|U(\tau)\|_{H^3}^2{\rm d}\tau\nonumber\\
&\lesssim (1+t)^{-\frac54}\|U_0\|_{H^3\cap L^1}+Y_\infty(t)\int_0^t
(1+t-\tau)^{-\frac74+\vartheta}(1+\tau)^{-\frac32}d\tau\nonumber\\
&\lesssim (1+t)^{-\frac54}\|U_0\|_{H^3\cap L^1}+Y_\infty(t)(1+t)^{-\frac32}.
\end{align}
Moreover, similar to \eqref{gron}, we have the following differential inequalities
\begin{equation}\label{G4.31}
\frac{{\rm d}}{{\rm d}t}\|U(t)\|_{H^3}^2
+\eta_{3}^\prime \|U(t)\|_{H^3}^2
\lesssim
\|U(t)\|_{L^2}^2,
\end{equation}
and
\begin{equation}\label{G4.32}
\frac{{\rm d}}{{\rm d}t}\|\nabla U(t)\|_{H^2}^2
+\eta_{4}^\prime \|\nabla U(t)\|_{H^2}^2
\lesssim
 \|\nabla U(t)\|_{L^2}^2,
\end{equation}
for some constants $\eta^\prime_{3},\eta_{4}^\prime>0$. Then applying Gronwall's inequality to \eqref{G4.31} and using \eqref{G4.34}, we obtain
\begin{align}
\|U(t)\|_{H^3}^2
\lesssim&\,
e^{-\lambda t}\|U_0\|_{H^3}^2
+
\int_0^t
e^{-\lambda(t-\tau)}
\|U(\tau)\|_{L^2}^2\,{\rm d}\tau
\nonumber\\
\lesssim&\,
\big(\|U_0\|_{H^3\cap L^1}^2+(Y_{\infty}(t))^2\big)
(1+t)^{-\frac32}. \label{G4.41}
\end{align}

Similarly, from \eqref{G4.37} and \eqref{G4.32}, we derive 
\begin{align} 
\|\nabla U(t)\|_{H^2}^2
\lesssim&\,
e^{-\lambda t}\|\nabla U_0\|_{H^2}^2+ \int_0^t
e^{-\lambda(t-\tau)}
\|\nabla U(\tau)\|_{L^2}^2 {\rm d}\tau
\nonumber\\
\lesssim&\,
\big(\|U_0\|_{H^3\cap L^1}^2+(Y_{\infty}(t))^2\big)
(1+t)^{-\frac52}.\label{G4.43}
\end{align}

For the estimate of \(\|u(t)-\omega(t)\|_{L^2}\), from \eqref{G4.10}, \eqref{G4.21} and \eqref{G4.28}, we get
\begin{align} \label{G4.39}
\|u(t)-\omega(t)\|_{L^2}
&\lesssim
(1+t)^{-\frac54}\|U_0\|_{H^3\cap L^1}
+
\int_0^t
(1+t-\tau)^{-\frac74+\vartheta}
\|U(\tau)\|_{H^3}^2{\rm d}\tau\nonumber\\
&\lesssim
\big(\|U_0\|_{H^3\cap L^1} +Y_{\infty}(t) \big)
(1+t)^{-\frac54}.
\end{align}

Thus, combining  \eqref{G4.41}, \eqref{G4.43} and \eqref{G4.39}, we arrive at
\begin{equation*} 
Y_{\infty}(t)
\lesssim
\|U_0\|_{H^3\cap L^1}^2+(Y_{\infty}(t))^2.
\end{equation*}
Taking \(\|U_0\|_{H^3\cap L^1}\) sufficiently small, we conclude that
\begin{equation*}
Y_{\infty}(t)
\lesssim
\|U_0\|_{H^3\cap L^1}^2,
\end{equation*}
for any $t\geq0$. This proves \eqref{TGG3} and \eqref{TGG5}.

Finally, it remains to establish \eqref{TGG4}. From \eqref{TGG5}, we obtain
\begin{align*} 
\|u(t)-\omega(t)\|_{H^3}
\lesssim \|\nabla (u,\omega)\|_{H^2}+\|u(t)-\omega(t)\|_{L^2}\lesssim
 (1+t)^{-\frac54}
\|U_0\|_{H^3\cap L^1}.
 \end{align*}
The proof of Theorem \ref{Th2} is therefore complete. \end{proof}

\subsection{Pressure estimates }

In this subsection, we establish some estimates of the pressure terms in \eqref{I4}, which will be used in the quantitative low Mach number limit later. Taking the divergence of  \eqref{I4}$_1$ and \eqref{I4}$_3$,
we can represent pressure terms $\pi_1$ and $\pi_2$ as follows:
\begin{align}\label{pi}
 -\Delta\pi_1=\partial_i\partial_j(u_i u_j),
 \quad
 -\Delta\pi_2=\partial_i\partial_j(\omega_i\omega_j).
 \end{align}
Then we can solve pressure functions $\pi_1$ and $\pi_2$ from \eqref{pi} as follows:
\begin{align}
\pi_1(t,x)=\mathcal{R}_i\mathcal{R}_j(u_iu_j)(t,x),\quad 
\pi_2(x)=\mathcal{R}_i\mathcal{R}_j(\omega_i\omega_j)(t,x),\label{pi1}
\end{align}
where the operators $\mathcal{R}_k$ with $k=1,2,3$ are the classical Riesz transformation in $\mathbb{R}^3.$

\begin{prop}[Pressure estimates]\label{Ppressure}
Under the assumptions of Theorem \ref{Th2}, the pressure functions $\pi_1$ and $\pi_2$ satisfy the following bounds:
\begin{align}\label{TGGG1}
\sup_{t\geq0}\|(\pi_1,\pi_2)(t)\|_{H^3}
\lesssim \|U_0\|_{H^3}
\end{align}
\begin{align}\label{TGGG2}
\int_0^\infty
\big(
\|\nabla(\pi_1,\pi_2)\|_{H^2}^2
+
\|(\nabla D_t^u\pi_1,\nabla D_t^\omega\pi_2)\|_{H^1}^2
\big)(\tau)\,{\rm d}\tau
\lesssim \|U_0\|_{H^3}^2,
\end{align}
and
\begin{align}\label{TGGG3}
\int_0^\infty
\|(D_t^u\pi_1,D_t^\omega\pi_2)(\tau)\|_{\dot H^{-1}}^2
\,{\rm d}\tau
\lesssim \|U_0\|_{H^3}^2,
\end{align}
where $D_t^u$ and $D_t^\omega$ are material derivatives defined by
\[
D_t^u:=\partial_t+u\cdot\nabla,
\quad
D_t^\omega:=\partial_t+\omega\cdot\nabla.
\]
\end{prop}

\begin{proof}

Applying the \(L^2\)-boundedness of Riesz transforms to \eqref{pi1}, we get
\begin{align*}
\sup_{t\geq 0}\|(\pi_1,\pi_{2})(t)\|_{H^3}\lesssim \sum_{i,j=1}^3\sup_{t\geq 0}\|(u_{i}u_{j},\omega_{i}\omega_{j})(t)\|_{H^3}\lesssim \sup_{t\geq 0} \|(u,\omega)(t)\|_{H^3}^2\lesssim  \|U_0\|_{H^3}, 
\end{align*}
which proves \eqref{TGGG1}. 

Next, we prove \eqref{TGGG2}. By a direct computation, we have
\begin{equation*}
\|\nabla(\pi_1,\pi_{2})\|_{H^2}
\lesssim
\|\nabla(u\otimes u)\|_{H^2}+\|\nabla(\omega\otimes\omega)\|_{H^2}
\lesssim
\|(u,\omega)\|_{H^3}\|\nabla (u,\omega)\|_{H^2},
\end{equation*}
 which, together with \eqref{TGG2}, yields
\begin{align}\label{G5.1}
\int_0^\infty \|\nabla(\pi_1,\pi_{2})(\tau)\|_{H^2}^2{\rm d}\tau  \lesssim \delta_1^2  \int_0^\infty  \|\nabla (u,\omega)(\tau)\|_{H^2}^2 {\rm d}\tau\lesssim \|U_0\|_{H^3}^2.
\end{align}
To estimate   the second integrand in the integral on the left-hand side of \eqref{TGGG2}, 
 we  use the boundedness of Riesz transforms  and product estimates to  get
\begin{align}\label{G5.2}
\|\nabla\partial_t(\pi_1,\pi_{2})\|_{H^1}
\lesssim&\,
\|\partial_t(u\otimes u)\|_{H^2}+\|\partial_t(\omega\otimes\omega)\|_{H^2}\nonumber\\
\lesssim&\, 
\|(u_t,\omega_{t})\|_{H^2}\|(u,\omega)\|_{H^3}. 
\end{align}
Moreover, we compute
\begin{align}\label{G5.3}
\|\nabla(u\cdot\nabla\pi_1)\|_{H^1}+\|\nabla(\omega\cdot\nabla\pi_2)\|_{H^1}
\lesssim&\,
\|u\cdot\nabla\pi_1\|_{H^2}+\|\omega\cdot\nabla\pi_2\|_{H^2}\nonumber\\
\lesssim&\, 
\|(u,\omega)\|_{H^3}\|\nabla(\pi_1,\pi_{2})\|_{H^2}.
\end{align}
Combining \eqref{G5.2} and \eqref{G5.3} yields
\begin{align}\label{G5.4}
\|\nabla (D_t^u\pi_1,D_t^\omega\pi_2)\|_{H^1} 
\lesssim&\,
\|(u,\omega)\|_{H^3}
\big(
\|(u_t,\omega_{t})\|_{H^2}
+
\|\nabla(\pi_1,\pi_{2})\|_{H^2}\big).
\end{align}
It suffices to estimate the time derivative $\|(u_{t},\omega_{t})\|_{H^2}$. Using the equations \eqref{I4}$_1$ and \eqref{I4}$_3$, we obtain
\begin{align}\label{G5.5}
\|(u_t,\omega_{t})\|_{H^2}
\lesssim&\,
\|\nabla u\|_{H^3}
+
\|(u,\omega)\|_{H^3}\|\nabla (u,\omega)\|_{H^2}
+
\|\nabla(\pi_1,\pi_{2})\|_{H^2}
+
\|u-\omega\|_{H^2}
\nonumber\\
\lesssim&\,
\|\nabla u\|_{H^3}
+
\|u-\omega\|_{H^2}
+\delta_1\|\nabla (u,\omega)\|_{H^2}+\|\nabla(\pi_1,\pi_{2})\|_{H^2}.
\end{align}
Plugging the estimate \eqref{G5.5} into \eqref{G5.4} yields
\begin{align*}
&\qquad\int_0^\infty\|(\nabla D_{t}^{u}\pi_1,\nabla D_{t}^{\omega}\pi_2)(\tau)\|_{H^1}^2 {\rm d}\tau\nonumber\\
&\quad\lesssim \,\delta_1^2\int_0^\infty  \big( \|\nabla{u}(\tau)\|_{H^3}^2    +\| ({u}-\omega)(\tau)\|_{H^2}^2+\|\nabla \omega(\tau)\|_{H^2}^2\big){\rm d}\tau \nonumber\\ &\qquad+\delta_1^2\int_0^\infty\|\nabla(\pi_1,\pi_2)(\tau)\|_{H^2}^2 {\rm d}\tau \nonumber\\
&\quad \lesssim  \|U_0\|_{H^3}^2.
\end{align*}
Combining this with \eqref{G5.1} yields \eqref{TGGG2}.

Finally, we prove \eqref{TGGG3}. By Lemma \ref{L2.3} and the Sobolev inequality $\|g\|_{\dot H^{-1}}\lesssim\|g\|_{L^{6/5}}$, we obtain
\begin{align}\label{G5.6}
\|\partial_t(\pi_1,\pi_{2})\|_{\dot H^{-1}}
\lesssim&\,
\|u_t\otimes u\|_{\dot H^{-1}}+\|\omega_t\otimes \omega\|_{\dot H^{-1}}\nonumber\\
\lesssim&\,
\|u_t u\|_{L^{\frac{6}{5}}}+\|\omega_t \omega\|_{L^{\frac{6}{5}}}\nonumber\\
\lesssim&\,
\|(u_t,\omega_{t})\|_{L^2}\|(u,\omega)\|_{L^3}.  
\end{align}
Similarly, we can control the transport term as follows:
\begin{align}\label{G5.7}
 \|(u\cdot\nabla\pi_1,\omega\cdot\nabla \pi_{2})\|_{\dot H^{-1}}
\lesssim\,
\|(u\cdot\nabla\pi_1,\omega\cdot\nabla\pi_2)\|_{L^{\frac65}}
\lesssim\,
\|(u,\omega)\|_{L^3}\|\nabla(\pi_1,\pi_2)\|_{L^2}.   
\end{align}
Then combining the estimates \eqref{G5.1}, \eqref{G5.5}, \eqref{G5.6} and \eqref{G5.7}, we have
\begin{align*}
&\int_0^\infty
\|(D_t^u\pi_1,D_t^\omega\pi_2)(\tau)\|_{\dot H^{-1}}^2
{\rm d}\tau\nonumber\\
&\qquad\lesssim
\sup_{\tau\geq0}\|(u,\omega)(\tau)\|_{H^3}^2
\int_0^\infty
\big(
\|(u_t,\omega_t)(\tau)\|_{L^2}^2
+
\|\nabla(\pi_1,\pi_2)(\tau)\|_{L^2}^2
\big)
 {\rm d}\tau\nonumber\\
 &\qquad\lesssim \|U_0\|_{H^3}^2,   
\end{align*}
which proves \eqref{TGGG3}. Therefore, the proof of Proposition \ref{Ppressure} is complete.
\end{proof}

\section{Quantitative low Mach number limit}

In this section, we prove Theorem \ref{Th3} on the convergence rate of low Mach number limit
to the scaled compressible two-phase  system \eqref{I3}. Let $(q^\varepsilon,u^\varepsilon, r^\varepsilon,\omega^\varepsilon)$ be the solution to the   problem \eqref{I3}--\eqref{I3-1}, and $(u,\pi_1,\omega,\pi_2)$ be the solution to the limiting two-phase problem \eqref{I4}--\eqref{I4-1}. We introduce the following perturbation
\begin{align*}
(\tilde{q},\tilde{u},\tilde{r},\tilde{\omega}):=\left(q^\varepsilon-\varepsilon [P'(1)]^{-1}\pi_1,u^\varepsilon-u, r^\varepsilon-\varepsilon\pi_2, \omega^\varepsilon-\omega \right).
\end{align*}
Then it is straightforward to check that  $(\tilde{q},\tilde{u},\tilde{r},\tilde{\omega})$ satisfies
\begin{equation}\label{G6.1}
\left\{
\begin{aligned}
&\partial_t\tilde{q}
+u^\varepsilon\cdot\nabla\tilde{q}
+\frac{1+\varepsilon q^\varepsilon}{\varepsilon}
\operatorname{div}\tilde{u}
=
\tilde{F}_1,
\\
&\partial_t\tilde{u}
+u^\varepsilon\cdot\nabla\tilde{u}
+\frac1{\varepsilon}
\frac{P'(1+\varepsilon q^\varepsilon)}
     {1+\varepsilon q^\varepsilon}
\nabla\tilde{q}
-\mu\Delta\tilde{u}
-(\mu+\lambda)\nabla\operatorname{div}\tilde{u}
-(\tilde{\omega}-\tilde{u})
=
\tilde{F}_2,
\\
&\partial_t\tilde{r}
+\omega^\varepsilon\cdot\nabla\tilde{r}
+\frac{1+\varepsilon r^\varepsilon}{\varepsilon}
\operatorname{div}\tilde{\omega}
=
\tilde{F}_3,
\\
&\partial_t\tilde{\omega}
+\omega^\varepsilon\cdot\nabla\tilde{\omega}
+\frac1{\varepsilon}
\frac1{1+\varepsilon r^\varepsilon}
\nabla\tilde{r}
-(\tilde{u}-\tilde{\omega})
=
\tilde{F}_4,
\end{aligned}
\right.
\end{equation}
where the source terms $\tilde{F}_i(i=1,2,3, 4)$ are given by
\begin{equation*}
\left\{
\begin{aligned}
\tilde{F}_1
:=&\,
-\varepsilon [P'(1)]^{-1}
\big(D_t^u\pi_1+\tilde{u}\cdot\nabla\pi_1\big),
\\
\tilde{F}_2
:=&\,
-\tilde{u}\cdot\nabla u
-\Big(
[P'(1)]^{-1}
\frac{P'(1+\varepsilon q^\varepsilon)}
     {1+\varepsilon q^\varepsilon}
-1
\Big)\nabla\pi_1-\frac{\mu\varepsilon q^\varepsilon}
      {1+\varepsilon q^\varepsilon}
\Delta{u}^\varepsilon
\\
&\,
-\frac{(\mu+\lambda)\varepsilon q^\varepsilon}
      {1+\varepsilon q^\varepsilon}
\nabla\operatorname{div}u^\varepsilon
+\frac{\varepsilon(r^\varepsilon-q^\varepsilon)}
       {1+\varepsilon q^\varepsilon}
(\omega-u)
+\frac{\varepsilon(r^\varepsilon-q^\varepsilon)}
       {1+\varepsilon q^\varepsilon}
(\tilde{\omega}-\tilde{u}),
\\
\tilde{F}_3
:=&\,
-\varepsilon
\big(D_t^\omega\pi_2+\tilde{\omega}\cdot\nabla\pi_2\big),
\\
\tilde{F}_4
:=&\,
-\tilde{\omega}\cdot\nabla\omega
+\frac{\varepsilon r^\varepsilon}
      {1+\varepsilon r^\varepsilon}
\nabla\pi_2 .
\end{aligned}
\right.
\end{equation*}

We now define the error functional $\tilde{\mathcal{X}}(t)$ as follows:
\begin{align}\label{G6.2}
\tilde{\mathcal{X}}(t):=&\,\sup_{\tau\in [0,t]}\|(\tilde{q},\tilde{u},\tilde{r},\tilde{\omega})(\tau)\|_{H^2}^2 +\int_0^t   \|\nabla\tilde{u}(\tau)\|_{H^2}^2 {\rm d}\tau \nonumber\\
 &+\int_0^t \big(\|\nabla (\tilde{q},\tilde{r})(\tau)\|_{H^1}^2+ \|(\tilde{u}-\tilde{\omega})(\tau)\|_{H^2}^2    \big) {\rm d}\tau.
\end{align}
Note that
\begin{align*}
\|\nabla \tilde{\omega}\|_{H^1}\lesssim \|(\tilde{u}-\tilde{\omega})\|_{H^2}+\|\nabla  \tilde{u}\|_{H^1}.
\end{align*}
Then we obtain
\begin{align*}
\int_0^t \|\nabla \tilde{\omega} (\tau)\|_{H^1}^2\,{\rm d}\tau \lesssim  \tilde{\mathcal{X}}(t).  
\end{align*}

First, we derive the zero-order estimate of $(\tilde{q},\tilde{u},\tilde{r},\tilde{\omega})$.
\begin{lem}\label{L6.1}
The solution $(\tilde{q},\tilde{u},\tilde{r},\tilde{\omega})$ to \eqref{G6.1} satisfies
\begin{align}\label{G6.3}
 &\sup_{0\leq \tau\leq t}  \big ( P^\prime(1) \|\tilde{q}(\tau)\|_{L^2}^2+ \|(\tilde{u},\tilde{r},\tilde{\omega})(\tau)\|_{L^2}^2  \big)+\int_0^t\big( \|\nabla\tilde{u}(\tau) \|_{L^2}^2
 + \|(\tilde{u}-\tilde{\omega})(\tau)\|_{L^2}^2\big
 )\,{\rm d}\tau\nonumber\\
 &\quad\leq \tilde{\mathcal{X }}(0)+C\varepsilon^2+C\big(\delta_0+\delta_1\big) \tilde{\mathcal{X}}(t),
\end{align}
where $\delta_0$ and $\delta_1$ are given in \eqref{TG1} and \eqref{TGG1} respectively.
\end{lem}

\begin{proof}
Multiplying \eqref{G6.1}$_1$--\eqref{G6.1}$_4$ by
$ P'(1)\tilde{q}, \tilde{u}, \tilde{r},$ and $ \tilde{\omega}$, respectively,  and integrating by parts, we obtain
\begin{align}\label{G6.4}
&\frac12\frac{\rm d}{{\rm d}t}
\big(
P'(1)\|\tilde{q}\|_{L^2}^2
+\|(\tilde{u},\tilde{r},\tilde{\omega})\|_{L^2}^2
\big)
+\mu\|\nabla\tilde{u}\|_{L^2}^2
+(\mu+\lambda)\|\operatorname{div}\tilde{u}\|_{L^2}^2
+\|\tilde{u}-\tilde{\omega}\|_{L^2}^2
\nonumber\\
=&\,
\frac{P'(1)}2\int_{\mathbb R^3}|\tilde{q}|^2\operatorname{div}u^\varepsilon {\rm d}x
+\frac12\int_{\mathbb R^3}|\tilde{u}|^2\operatorname{div}u^\varepsilon {\rm d}x
+\frac12\int_{\mathbb R^3}|\tilde{r}|^2\operatorname{div}\omega^\varepsilon\,{\rm d}x
+\frac12\int_{\mathbb R^3}|\tilde{\omega}|^2\operatorname{div}\omega^\varepsilon\,{\rm d}x
\nonumber\\
&
-\frac1\varepsilon\int_{\mathbb R^3}
\Big(
P'(1)(1+\varepsilon q^\varepsilon)
-\frac{P'(1+\varepsilon q^\varepsilon)}
      {1+\varepsilon q^\varepsilon}
\Big)\tilde{q} \operatorname{div}\tilde{u}{\rm d}x
+\frac1\varepsilon\int_{\mathbb R^3}
\tilde{q} \nabla\Big(
\frac{P'(1+\varepsilon q^\varepsilon)}
     {1+\varepsilon q^\varepsilon}
\Big)\cdot\tilde{u} {\rm d}x
\nonumber\\
&
-\frac1\varepsilon\int_{\mathbb R^3}
\Big(
1+\varepsilon r^\varepsilon-\frac1{1+\varepsilon r^\varepsilon}
\Big)\tilde{r}\operatorname{div}\tilde{\omega} {\rm d}x
+\frac1\varepsilon\int_{\mathbb R^3}
\tilde{r} \nabla\Big(
\frac1{1+\varepsilon r^\varepsilon}
\Big)\cdot\tilde{\omega} {\rm d}x
\nonumber\\
&
-\varepsilon\int_{\mathbb R^3}
\big(
\tilde{q}\,D_t^u\pi_1+\tilde{r}D_t^\omega\pi_2
\big) {\rm d}x
-\varepsilon\int_{\mathbb R^3}
\big(
\tilde{q} \tilde{u}\cdot\nabla\pi_1
+\tilde{r} \tilde{\omega}\cdot\nabla\pi_2
\big) {\rm d}x-\int_{\mathbb R^3}
(\tilde{u}\cdot\nabla u)\cdot\tilde{u} {\rm d}x
\nonumber\\
&
-\int_{\mathbb R^3}
(\tilde{\omega}\cdot\nabla\omega)\cdot\tilde{\omega} {\rm d}x
-\int_{\mathbb R^3}
\Big(
[P'(1)]^{-1}
\frac{P'(1+\varepsilon q^\varepsilon)}
     {1+\varepsilon q^\varepsilon}
-1
\Big)\nabla\pi_1\cdot\tilde{u} {\rm d}x
+\int_{\mathbb R^3}
\frac{\varepsilon r^\varepsilon}{1+\varepsilon r^\varepsilon}
\nabla\pi_2\cdot\tilde{\omega}{\rm d}x
\nonumber\\
&-\int_{\mathbb R^3}
\Big(
\frac{\mu\varepsilon q^\varepsilon}
     {1+\varepsilon q^\varepsilon}
\Delta u^\varepsilon
+
\frac{(\mu+\lambda)\varepsilon q^\varepsilon}
     {1+\varepsilon q^\varepsilon}
\nabla\operatorname{div}u^\varepsilon
\Big)\cdot\tilde{u}{\rm d}x
+\int_{\mathbb R^3}
\frac{\varepsilon(r^\varepsilon-q^\varepsilon)}
     {1+\varepsilon q^\varepsilon}
(\omega-u)\cdot\tilde{u} {\rm d}x
\nonumber\\
& 
+\int_{\mathbb R^3}
\frac{\varepsilon(r^\varepsilon-q^\varepsilon)}
     {1+\varepsilon q^\varepsilon}
(\tilde{\omega}-\tilde{u})\cdot \tilde{u} {\rm d}x
\equiv:\,
\sum_{i=1}^{17}\tilde{ I}_i.
\end{align}
Now we estimate the right-hand side of \eqref{G6.4}. 

First, for $\tilde{I}_1,\dots,\tilde{I}_4$,  we use  \eqref{TG2} and \eqref{TGG2} to obtain
\begin{align}\label{G6.5}
\sum_{i=1}^4\int_0^t|\tilde{I}_i|\,{\rm d}\tau
\lesssim&\,
\int_0^t
\|\nabla(u^\varepsilon,\omega^\varepsilon)\|_{H^2}
\|(\tilde{q}, \tilde{u}, \tilde{r}, \tilde{\omega}))\|_{H^1}
\|\nabla(\tilde{q}, \tilde{u}, \tilde{r}, \tilde{\omega})\|_{L^2}{\rm d}\tau
\nonumber\\
\lesssim&\,\|(\tilde{q}, \tilde{u}, \tilde{r}, \tilde{\omega})\|_{L_t^\infty(H^1)} \|\nabla(\tilde{q}, \tilde{u}, \tilde{r}, \tilde{\omega})\|_{L_t^2(L^2)}  \|\nabla(u^\varepsilon,\omega^\varepsilon)\|_{L_t^2(H^2)}        \nonumber \\
\lesssim&\,
\delta_0\tilde{\mathcal X}(t).
\end{align}
For the singular terms $\tilde{ I}_5, \dots, \tilde{ I}_8$, we observe that
\begin{align*}
\frac1\varepsilon
\Big|
P'(1)(1+\varepsilon q^\varepsilon)
-\frac{P'(1+\varepsilon q^\varepsilon)}
      {1+\varepsilon q^\varepsilon}
\Big|
\lesssim&  |q^\varepsilon|,  \qquad
\frac1\varepsilon
\Big|
1+\varepsilon r^\varepsilon
-\frac1{1+\varepsilon r^\varepsilon}
\Big|
\lesssim  |r^\varepsilon|,
\end{align*}
and
\begin{align*}
\frac1\varepsilon
\Big|
\nabla\Big(
\frac{P'(1+\varepsilon q^\varepsilon)}
     {1+\varepsilon q^\varepsilon}
\Big)\Big|
\lesssim |\nabla q^\varepsilon|,
\qquad
\frac1\varepsilon
\Big|
\nabla\Big(
\frac1{1+\varepsilon r^\varepsilon}
\Big)\Big|
\lesssim |\nabla r^\varepsilon|.    
\end{align*}
Utilizing these bounds, the terms  $\tilde{I}_5, \dots, \tilde{I}_8$ can be estimated as 
\begin{align}\label{G6.6}
\sum_{i=5}^8\int_0^t|\tilde{I}_i| {\rm d}\tau
\lesssim&\,
\int_0^t
\|(q^\varepsilon,r^\varepsilon)\|_{H^2}
\|\nabla(\tilde{q},\tilde{ r})\|_{L^2}
\|\nabla(\tilde{ u},\tilde{\omega})\|_{L^2} {\rm d}\tau
\nonumber\\
\lesssim&\, \|(q^\varepsilon,r^\varepsilon)\|_{L_t^\infty(H^2)}  \|\nabla(\tilde{q},\tilde {r})\|_{L_t^2(L^2)} \|\nabla(\tilde{ u},\tilde{\omega})\|_{L_t^2(L^2)}     \nonumber\\
\lesssim&\,
\delta_0\tilde{\mathcal X}(t).
\end{align}
For the material derivative term \(\tilde{I}_9\),  we use   Proposition \ref{Ppressure} to obtain
\begin{align}\label{G6.7}
\int_0^t |\tilde{I}_9| {\rm d}\tau
\lesssim&\,
\varepsilon
\int_0^t
\big(
\|D_t^u\pi_1\|_{\dot H^{-1}}\|\nabla\tilde{ q}\|_{L^2}
+
\|D_t^\omega\pi_2\|_{\dot H^{-1}}\|\nabla\tilde{ r}\|_{L^2}
\big) {\rm d}\tau
\nonumber\\
\leq&\,   \varepsilon \big(\|D_{t}^u\pi_1\|_{L_t^2(\dot H^{-1})} \|\nabla\tilde{ q}\|_{L_t^2(L^2)}+\|D_{t}^\omega\pi_2\|_{L_t^2(\dot H^{-1})} \|\nabla\tilde{ r}\|_{L_t^2(L^2)} \big) \nonumber\\
\leq&\,
C\varepsilon^2
+
C\delta_1\tilde{\mathcal X}(t).
\end{align}
For \(\tilde{ I}_{10}\), we use the pressure estimate \eqref{TGGG2} to obtain
\begin{align}\label{G6.8}
\int_0^t |\tilde{ I}_{10}| {\rm d}\tau
\lesssim&\,
\varepsilon
\int_0^t
(
\|\tilde{ q}\|_{L^3}\|\tilde{ u}\|_{L^6}\|\nabla\pi_1\|_{L^2}
+
\|\tilde{ r}\|_{L^3}\|\tilde{\omega}\|_{L^6}\|\nabla\pi_2\|_{L^2}
 ) {\rm d}\tau
\nonumber\\
\lesssim&\, \varepsilon \|(\tilde{ q},\tilde{ r})\|_{L_t^\infty(H^1)} \|\nabla(\pi_1,\pi_2)\|_{L_t^2(L^2)}\|\nabla(\tilde{ u},\tilde{\omega})\|_{L_t^2(L^2)}  \nonumber\\
\lesssim&\,
\delta_1\tilde{\mathcal X}(t).
\end{align}
For the convection terms  \(\tilde{I}_{11}\) and \(\tilde{I}_{12}\), we obtain
\begin{align}\label{G6.9}
\int_0^t \big(|\tilde{ I}_{11}|+|\tilde{ I}_{12}|\big){\rm d}\tau
\lesssim&\,
\int_0^t
\big(
\|\tilde{ u}\|_{L^3}\|{ u}\|_{L^6}\|\nabla\tilde{ u}\|_{L^2}
+
\|\tilde{\omega}\|_{L^3}\|{\omega}\|_{L^6}\|\nabla\tilde{\omega}\|_{L^2}
\big){\rm d}\tau
\nonumber\\
\lesssim&\,  \|(\tilde{ u},\tilde{\omega})\|_{L_t^\infty(H^1)}\|\nabla(u,\omega)\|_{L_t^2(L^2)} \|\nabla(\tilde{ u},\tilde{\omega})\|_{L_t^2(L^2)}           \nonumber\\
\lesssim&\,
\delta_1\tilde{\mathcal X}(t).
\end{align}
For the pressure terms \(\tilde{ I}_{13}\) and \(\tilde{ I}_{14}\), note that
\begin{align*}
 \Big|
[P'(1)]^{-1}
\frac{P'(1+\varepsilon q^\varepsilon)}
     {1+\varepsilon q^\varepsilon}
-1
\Big|
\lesssim
\varepsilon |q^\varepsilon|,
\qquad
\Big|
\frac{\varepsilon r^\varepsilon}{1+\varepsilon r^\varepsilon}
\Big|
\lesssim
\varepsilon |r^\varepsilon|.
\end{align*}
Then it holds that
\begin{align}\label{G6.10}
\int_0^t\big(|\tilde{ I}_{13}|+|\tilde{I}_{14}|\big) {\rm d}\tau
\lesssim&\,
\varepsilon
\int_0^t
 \big(
\|q^\varepsilon\|_{L^3}\|\nabla\pi_1\|_{L^2}\|\tilde{ u}\|_{L^6}
+
\|r^\varepsilon\|_{L^3}\|\nabla\pi_2\|_{L^2}\|\tilde{\omega}\|_{L^6}
 \big){\rm d}\tau
\nonumber\\
\lesssim&\, \varepsilon \|(q^\varepsilon,r^\varepsilon)\|_{L_t^\infty(H^1)} \|\nabla(\tilde{ u},\tilde{\omega})\|_{L_t^2(L^2)}   \|\nabla(\pi_1,\pi_2)\|_{L_t^2(L^2)} \nonumber\\
\lesssim&\, \delta_1^2\tilde{\mathcal{X}}(t)  +
\delta_0^2\varepsilon^2 .
\end{align}

Fo the viscosity term \(\tilde{ I}_{15}\), we integrate by parts to obtain
\begin{align}\label{G6.11}
\int_0^t|\tilde{ I}_{15}|{\rm d}\tau
\lesssim&\,
\varepsilon
\int_0^t
\|q^\varepsilon\|_{H^2}
\|\nabla u^\varepsilon\|_{H^2}
\|\nabla\tilde{u}\|_{L^2}{\rm d}\tau
\nonumber\\
\lesssim&\,\varepsilon\|q^\varepsilon\|_{L_t^\infty(H^2)}\|\nabla u^\varepsilon\|_{L_t^2(H^2)} \|\nabla\tilde{ u}\|_{L_t^2(L^2)}\nonumber\\
\lesssim&\,\delta_0\tilde{\mathcal{X}}(t) +\delta_0\varepsilon^2.
\end{align}
Finally, for the drag terms $\tilde{I}_{16}$ and $\tilde{ I}_{17}$, we obtain
\begin{align}\label{G6.13}
\int_0^t|\tilde{I}_{16}|+|\tilde{I}_{17}|{\rm d}\tau
\lesssim&\,
\varepsilon
\int_0^t
\|r^\varepsilon-q^\varepsilon\|_{L^3}
\|\tilde{ u}\|_{L^6}\left(\|u-\omega\|_{L^2}+\|\tilde{u}-\tilde{\omega}\|_{L^2}\right){\rm d}\tau
\nonumber\\
\lesssim&\, \varepsilon \|(q^\varepsilon,r^\varepsilon)\|_{L_t^\infty(H^1)} \|\nabla\tilde{ u}\|_{L_t^2(L^2)}\left(\|u-\omega\|_{L_t^2(L^2)}+\|\tilde{u}-\tilde{\omega}\|_{L^2_t(L^2)}\right) \nonumber\\
\lesssim&\, \delta_0\tilde{\mathcal{X}}(t)+\delta_0 \varepsilon^2.
\end{align}

Substituting the bounds \eqref{G6.5}--\eqref{G6.13} into \eqref{G6.4} and integrating the resulting inequality over \([0,t]\), we derive \eqref{G6.3}. The proof of Lemma \ref{L6.1} is complete.
\end{proof}

Next, we derive the higher-order estimates on
\((\tilde{q},\tilde{ u},\tilde{ r},\tilde{\omega})\).
\begin{lem}\label{L6.2}
The solution \((\tilde{q},\tilde{ u},\tilde{ r},\tilde{\omega})\) to \eqref{G6.1} satisfies the following bounds:
\begin{align}\label{G6.14}
&\sup_{0\leq \tau\leq t}\sum_{1\leq|\alpha|\leq 2}
\bigg(
\bigg\|
\sqrt{\frac{P'(1+\varepsilon q^\varepsilon)}
     {1+\varepsilon q^\varepsilon}}
\partial^\alpha\tilde{q}(\tau)
\bigg\|_{L^2}^2
+\Big\|\frac{\partial^\alpha \tilde{ r}(\tau)}{\sqrt{1+\varepsilon r^\varepsilon}} \Big\|_{L^2}^2+
\|(\partial^\alpha\tilde{u},\partial^\alpha\tilde{\omega})(\tau)\|_{L^2}^2
\bigg)
\nonumber\\
&
\quad+
\sum_{1\leq|\alpha|\leq 2}\int_0^t
\big(
\|\nabla\partial^\alpha\tilde{ u}(\tau)\|_{L^2}^2
+
\|(\partial^\alpha\tilde{ u}-\partial^\alpha\tilde{\omega})(\tau)\|_{L^2}^2
\big){\rm d}\tau
\nonumber\\
&
\qquad\leq
C\tilde{\mathcal X}(0)
+
C\varepsilon^2
+
C\big(\delta_0+\delta_1\big)
\tilde{\mathcal X}(t).
\end{align}
\end{lem}

\begin{proof}
Applying \(\partial^\alpha\) to
\eqref{G6.1}$_1$--\eqref{G6.1}$_4$, and taking inner produce of resulting equations with $\big(\frac{P'(1+\varepsilon q^\varepsilon)}{1+\varepsilon q^\varepsilon} \partial^\alpha\tilde{ q},
\partial^\alpha\tilde{u},
\frac{\partial^\alpha\tilde{ r}}{1+\varepsilon r^\varepsilon},
,\partial^\alpha\tilde{\omega}\big)$,  we obtain
\begin{align}\label{G6.15}
&\frac12\frac{{\rm d}}{{\rm d}t}
\bigg(
\Big\|
\sqrt{\frac{P'(1+\varepsilon q^\varepsilon)}
     {1+\varepsilon q^\varepsilon}}
\partial^\alpha\tilde{ q}
\Big\|_{L^2}^2
+
\|\partial^\alpha\tilde{ u}\|_{L^2}^2
+
\Big\|
\frac{\partial^\alpha\tilde{ r}}
     {\sqrt{1+\varepsilon r^\varepsilon}}
\Big\|_{L^2}^2
+
\|\partial^\alpha\tilde{\omega}\|_{L^2}^2
\bigg)\nonumber\\
&\qquad+\mu\|\nabla\partial^\alpha\tilde{ u}\|_{L^2}^2
+(\mu+\lambda)\|\operatorname{div}\partial^\alpha\tilde{ u}\|_{L^2}^2
+\|\partial^\alpha(\tilde{ u}-\tilde{\omega})\|_{L^2}^2=\sum_{i=1}^8\tilde{J}_{i},
\end{align}
where
\begin{align}
\tilde{J}_1&=\frac12\int_{\mathbb R^3}
\partial_t\bigg(
\frac{P'(1+\varepsilon q^\varepsilon)}
     {1+\varepsilon q^\varepsilon}
\bigg)
|\partial^\alpha\tilde{q}|^2{\rm d}x
+
\frac12\int_{\mathbb R^3}
\partial_t\bigg(
\frac1{1+\varepsilon r^\varepsilon}
\bigg)
|\partial^\alpha\tilde{ r}|^2{\rm d}x,\nonumber\\
\tilde{J}_2&=-
\int_{\mathbb R^3}
\frac{P'(1+\varepsilon q^\varepsilon)}
     {1+\varepsilon q^\varepsilon}
\Big(
[\partial^\alpha,u^\varepsilon\cdot\nabla]\tilde{q}
+
[\partial^\alpha,q^\varepsilon\operatorname{div}]\tilde{ u}
\Big)
\partial^\alpha\tilde{ q}+[\partial^\alpha,u^\varepsilon\cdot \nabla]\tilde{u}\cdot \partial^\alpha\tilde{u}{\rm d}x
\nonumber\\
&\quad-
\int_{\mathbb R^3}
\frac{\partial^\alpha\tilde{r}}{1+\varepsilon r^\varepsilon}
\Big(
[\partial^\alpha,\omega^\varepsilon\cdot\nabla]\tilde{ r}
+
[\partial^\alpha,r^\varepsilon\operatorname{div}]\tilde{\omega}
\Big)+[\partial^\alpha,\omega^\varepsilon\cdot \nabla]\tilde{\omega}\cdot \partial^\alpha\tilde{\omega}{\rm d}x,\nonumber\\
\tilde{J}_3&=\frac12\int_{\mathbb R^3}
\operatorname{div}\bigg(
\frac{P'(1+\varepsilon q^\varepsilon)}
     {(1+\varepsilon q^\varepsilon)^2}
u^\varepsilon
\bigg)
|\partial^\alpha\tilde{ q}|^2{\rm d}x
+
\frac12\int_{\mathbb R^3}
\operatorname{div}\bigg(
\frac{\omega^\varepsilon}
     {(1+\varepsilon r^\varepsilon)^2}
\bigg)
|\partial^\alpha\tilde{ r}|^2{\rm d}x\nonumber\\
&\quad+\frac12\int_{\mathbb R^3}
|\partial^\alpha\tilde{ u}|^2
\operatorname{div}u^\varepsilon{\rm d}x+\frac12\int_{\mathbb R^3}
|\partial^\alpha\tilde{\omega}|^2
\operatorname{div}\omega^\varepsilon{\rm d}x,\nonumber\\
\tilde{J}_4&=\frac1\varepsilon\int_{\mathbb R^3}
\nabla\bigg(
\frac{P'(1+\varepsilon q^\varepsilon)}
     {1+\varepsilon q^\varepsilon}
\bigg)\partial^\alpha\tilde{ u}
\partial^\alpha\tilde{ q}
{\rm d}x+\frac1\varepsilon\int_{\mathbb R^3}
\nabla\bigg(
\frac1{1+\varepsilon r^\varepsilon}
\bigg)
\partial^\alpha\tilde{ r}
\partial^\alpha\tilde{\omega}{\rm d}x\nonumber\\
&\quad -
\frac{1}{\varepsilon}\int_{\mathbb R^3}
\Big[
\partial^\alpha,
\frac{P'(1+\varepsilon q^\varepsilon)}
     {1+\varepsilon q^\varepsilon}
\Big]\nabla\tilde{ q}
\cdot\partial^\alpha\tilde{ u}{\rm d}x-\frac{1}{\varepsilon}\int_{\mathbb R^3}
\Big[
\partial^\alpha,
\frac1{1+\varepsilon r^\varepsilon}
\Big]\nabla\tilde{ r}
\cdot\partial^\alpha\tilde{\omega}{\rm d}x,\nonumber\\
\tilde{J}_5&=\int_{\mathbb R^3}
\frac{P'(1+\varepsilon q^\varepsilon)}
     {(1+\varepsilon q^\varepsilon)^2}
\partial^\alpha\tilde{ F}_1
\partial^\alpha\tilde{ q}{\rm d}x,~\tilde{J}_6=\int_{\mathbb R^3}
\partial^\alpha\tilde{ F}_2\cdot
\partial^\alpha\tilde{ u}{\rm d}x,\nonumber\\
\tilde{J}_7&=\int_{\mathbb R^3}
\frac1{(1+\varepsilon r^\varepsilon)^2}
\partial^\alpha\tilde{ F}_3
\partial^\alpha\tilde{ r}{\rm d}x,~\tilde{J}_8=\int_{\mathbb R^3}
\partial^\alpha\tilde{ F}_4\cdot
\partial^\alpha\tilde{\omega}{\rm d}x.\nonumber
\end{align}

We now estimate the right-hand side of \eqref{G6.15}.  Adapting the same approach as in Lemma \ref{L3.2}, we   get the following estimates of $\tilde{J}_1,\dots,\tilde{J}_4$:
\begin{align}\label{j1-4}
\sum_{i=1}^4\int_0^t\tilde{J}_i{\rm d}\tau\lesssim\delta_0\tilde{\mathcal X}(t).
\end{align}
For the terms \(\tilde{ J}_{5}\) and
\(\tilde{ J}_{7} \), we use the pressure estimates in Proposition \ref{Ppressure} to obtain
\begin{align}\label{G6.19}
&\int_0^t
 (|\tilde{J}_{5} |+|\tilde{ J}_{7}| ){\rm d}\tau
\nonumber\\
\lesssim&\,
\varepsilon
\int_0^t
\left(\|\nabla(D_t^u\pi_1,D_t^\omega\pi_2)\|_{H^1}+\|(\tilde{ u},\tilde{\omega})\|_{H^2}
\|\nabla(\pi_1,\pi_2)\|_{H^2}\right)
\|\nabla(\tilde{ q},\tilde{ r})\|_{H^1}{\rm d}\tau
\nonumber\\
\lesssim&\,
\varepsilon
\left(\|\nabla(D_t^u\pi_1,D_t^\omega\pi_2)\|_{L_t^2(H^1)}+\|(\tilde{u},\tilde{\omega})\|_{L_t^\infty(H^2)}
\|\nabla(\pi_1,\pi_2)\|_{L_t^2(H^2)}\right)
\|\nabla(\tilde{ q},\tilde{ r})\|_{L_t^2(H^1)}
\nonumber\\
\lesssim&\,
\varepsilon^2
+
\delta_1^2\tilde{\mathcal X}(t).
\end{align}

We now turn to estimate \(\tilde{ J}_{6}\) and
\(\tilde{ J}_{8} \). The convection terms can be bounded as follows:
\begin{align}\label{G6.20}
&\int_0^t
\Big|
\int_{\mathbb R^3}
\partial^\alpha(\tilde{ u}\cdot\nabla u)\cdot
\partial^\alpha\tilde{ u}{\rm d}x
\Big|{\rm d}\tau
+
\int_0^t
\Big|
\int_{\mathbb R^3}
\partial^\alpha(\tilde{\omega}\cdot\nabla\omega)\cdot
\partial^\alpha\tilde{\omega}{\rm d}x
\Big|{\rm d}\tau
\nonumber\\
 \lesssim&\,
\int_0^t
\|\nabla(u,\omega)\|_{H^2}
\|(\tilde{ u},\tilde{\omega})\|_{H^2}
\|\nabla(\tilde{ u},\tilde{\omega})\|_{H^1}{\rm d}\tau
\nonumber\\
 \lesssim&\,
\|\nabla(u,\omega)\|_{L_t^2(H^2)}
\|(\tilde{u},\tilde{\omega})\|_{L_t^\infty(H^2)}
\|\nabla(\tilde{ u},\tilde{\omega})\|_{L_t^2(H^1)}
\lesssim
\delta_1\tilde{\mathcal X}(t).
\end{align}
Using Lemma \ref{L2.1}, we   obtain
\begin{align}\label{G6.21}
&\int_0^t
\Big|
\int_{\mathbb R^3}
\partial^\alpha\bigg(
\Big(
[P'(1)]^{-1}
\frac{P'(1+\varepsilon q^\varepsilon)}
     {1+\varepsilon q^\varepsilon}
-1
\Big)\nabla\pi_1
\bigg)\cdot\partial^\alpha\tilde{ u}{\rm d}x
\Big|{\rm d}\tau
\nonumber\\
&+
\int_0^t
\Big|
\int_{\mathbb R^3}
\partial^\alpha\Big(
\frac{\varepsilon r^\varepsilon}
     {1+\varepsilon r^\varepsilon}
\nabla\pi_2
\Big)\cdot\partial^\alpha\tilde{\omega}{\rm d}x
\Big|{\rm d}\tau
\nonumber\\
\lesssim&\,
\varepsilon
\int_0^t
\|(q^\varepsilon,r^\varepsilon)\|_{H^2}
\|\nabla(\pi_1,\pi_2)\|_{H^2}
\|\nabla(\tilde{u},\tilde{\omega})\|_{H^1}{\rm d}\tau
\nonumber\\
\lesssim&\,
\varepsilon
\|(q^\varepsilon,r^\varepsilon)\|_{L_t^\infty(H^2)}
\|\nabla(\pi_1,\pi_2)\|_{L_t^2(H^2)}
\|\nabla(\tilde{ u},\tilde{\omega})\|_{L_t^2(H^1)}
\nonumber\\
\lesssim&\,
 \varepsilon^2
+
 \delta_1\tilde{\mathcal X}(t).
\end{align}
For the terms involving viscosity, we integrate by parts to obtain
\begin{align}\label{G6.22}
 &\int_0^t
\Big|
\int_{\mathbb R^3}
\partial^\alpha\Big(
\frac{\mu\varepsilon q^\varepsilon}
     {1+\varepsilon q^\varepsilon}
\Delta u^\varepsilon
+
\frac{(\mu+\lambda)\varepsilon q^\varepsilon}
     {1+\varepsilon q^\varepsilon}
\nabla\operatorname{div}u^\varepsilon
\Big)\cdot
\partial^\alpha\tilde{ u}{\rm d}x
\Big|{\rm d}\tau\nonumber\\
\lesssim&\,
\varepsilon
\int_0^t
\|q^\varepsilon\|_{H^3}
\|\nabla u^\varepsilon\|_{H^3}\bigg(\sum_{\alpha=1,2}
\|\nabla\partial^\alpha\tilde{ u}\|_{L^2}\bigg){\rm d}\tau
\nonumber\\
 \lesssim& \,
\varepsilon
\|q^\varepsilon\|_{L_t^\infty(H^3)}
\|\nabla u^\varepsilon\|_{L_t^2(H^3)}\bigg(\sum_{\alpha=1,2}
\|\nabla\partial^\alpha\tilde{ u}\|_{L^2_t(L^2)}\bigg)
\nonumber\\
 \lesssim&\,
\varepsilon^2+\delta_0\tilde{\mathcal X}(t).
\end{align}
For the drag terms, we use Lemma \ref{L2.1} again to obtain
\begin{align}\label{G6.23}
& \varepsilon\int_0^t
\Big|
\int_{\mathbb R^3}
\partial^\alpha\Big(
\frac{r^\varepsilon-q^\varepsilon}
     {1+\varepsilon q^\varepsilon}
(\omega-u+\tilde{\omega}-\tilde{u})
\Big)\cdot
\partial^\alpha\tilde{ u}{\rm d}x
\Big|{\rm d}\tau\nonumber\\
\lesssim &\,
\varepsilon
\int_0^t
\|(q^\varepsilon,r^\varepsilon)\|_{H^2}
\left(\|\omega-u\|_{H^2}+\|\tilde{\omega}-\tilde{u}\|_{H^2}\right)
\|\nabla\tilde{ u}\|_{H^1}{\rm d}\tau
\nonumber\\
 \lesssim&\,
\varepsilon
\|(q^\varepsilon,r^\varepsilon)\|_{L_t^\infty(H^2)}
\left(\|\omega-u\|_{L_t^2(H^2)}+\|\tilde{\omega}-\tilde{u}\|_{L^2_t(H^2)}\right)
\|\nabla\tilde{ u}\|_{L_t^2(H^1)}
\nonumber\\
 \lesssim&\,
 \varepsilon^2
+
 \Big(\delta_0+\delta_1\Big)
\tilde{\mathcal X}(t),
\end{align}
Collecting the estimates  \eqref{G6.20}--\eqref{G6.23} together, we deduce that
\begin{align}\label{G6.25}
\int_0^t
 (
|\tilde{ J}_{6} |+
|\tilde{ J}_{8} |
 ){\rm d}\tau
\lesssim&\,
 \varepsilon^2
+
 \Big(\delta_0+\delta_1\Big)
\tilde{\mathcal X}(t).
\end{align}

Substituting the bounds \eqref{j1-4}, \eqref{G6.19} and \eqref{G6.25} into \eqref{G6.15}, we obtain \eqref{G6.14}. The proof of Lemma \ref{L6.2} is complete.
\end{proof}

Finally, we recover the dissipation of \(\tilde{ q}\) and \(\tilde{ r}\).

\begin{lem}\label{L6.3}
It holds that
\begin{align}\label{G6.27}
\int_0^t
\|\nabla(\tilde{ q},\tilde{ r})(\tau)\|_{H^1}^2{\rm d}\tau
\lesssim
\tilde{\mathcal X}(0)
+
\varepsilon^2
+
\big(\delta_0+\delta_1\big)
\tilde{\mathcal X}(t).
\end{align}
\end{lem}

\begin{proof}
We first estimate the dissipation of \(\tilde{ q}\).
Applying \(\partial^\alpha\) with $\alpha=0$ or $1$ to \eqref{G6.1}$_2$, and taking inner product of the resultant equation with 
\(\varepsilon\nabla\partial^\alpha\tilde{ q}\), we obtain
\begin{align}\label{G6.28}
&\int_{\mathbb R^3}
\frac{P'(1+\varepsilon q^\varepsilon)}
     {1+\varepsilon q^\varepsilon}
|\nabla\partial^\alpha\tilde{ q}|^2{\rm d}x
\nonumber\\
=&\,
-\varepsilon
\int_{\mathbb R^3}
\nabla\partial^\alpha\tilde{ q}\cdot
\partial^\alpha\partial_t\tilde{ u}{\rm d}x
-\varepsilon
\int_{\mathbb R^3}
\nabla\partial^\alpha\tilde{ q}\cdot
\partial^\alpha\big(u^\varepsilon\cdot\nabla\tilde{ u}\big){\rm d}x+
\varepsilon\mu
\int_{\mathbb R^3}
\nabla\partial^\alpha\tilde{ q}\cdot
\Delta\partial^\alpha\tilde{ u}{\rm d}x
\nonumber\\
&
+
\varepsilon(\mu+\lambda)
\int_{\mathbb R^3}
\nabla\partial^\alpha\tilde{ q}\cdot
\nabla\operatorname{div}\partial^\alpha\tilde{ u}{\rm d}x+
\varepsilon
\int_{\mathbb R^3}
\nabla\partial^\alpha\tilde{ q}\cdot
\partial^\alpha(\tilde{\omega}-\tilde{ u}){\rm d}x
\nonumber\\
&
-\int_{\mathbb R^3}
\nabla\partial^\alpha\tilde{ q}\cdot
\Big[
\partial^\alpha,
\frac{P'(1+\varepsilon q^\varepsilon)}
     {1+\varepsilon q^\varepsilon}
\Big]\nabla\tilde{ q}{\rm d}x+
\varepsilon
\int_{\mathbb R^3}
\nabla\partial^\alpha\tilde{ q}\cdot
\partial^\alpha\tilde{ F}_2{\rm d}x
\nonumber\\
\equiv:&\,
\sum_{j=1}^{7}\tilde{ K}_j .
\end{align}
We first deal with  \(\tilde{ K}_1\). Integrating by parts and using
\eqref{G6.1}$_1$, we have
\begin{align*} 
\tilde{ K}_1
=&\,
-\varepsilon\frac{{\rm d}}{{\rm d}t}
\int_{\mathbb R^3}
\nabla\partial^\alpha\tilde{ q}\cdot
\partial^\alpha\tilde{ u}{\rm d}x
+
\varepsilon
\int_{\mathbb R^3}
\nabla\partial^\alpha\partial_t\tilde{ q}\cdot
\partial^\alpha\tilde{ u}{\rm d}x
\nonumber\\
=&\,
-\varepsilon\frac{{\rm d}}{{\rm d}t}
\int_{\mathbb R^3}
\nabla\partial^\alpha\tilde{ q}\cdot
\partial^\alpha\tilde{ u}{\rm d}x
-
\varepsilon
\int_{\mathbb R^3}
\partial^\alpha\partial_t\tilde{ q}\,
\operatorname{div}\partial^\alpha\tilde{ u}{\rm d}x
\nonumber\\
=&\,
-\varepsilon\frac{{\rm d}}{{\rm d}t}
\int_{\mathbb R^3}
\nabla\partial^\alpha\tilde{ q}\cdot
\partial^\alpha\tilde{ u}{\rm d}x
+
\int_{\mathbb R^3}
\partial^\alpha\big((1+\varepsilon q^\varepsilon)\operatorname{div}\tilde{ u}\big)
\operatorname{div}\partial^\alpha\tilde{ u}{\rm d}x
\nonumber\\
&\,+
\varepsilon
\int_{\mathbb R^3}
\partial^\alpha\big(u^\varepsilon\cdot\nabla\tilde{ q}\big)
\operatorname{div}\partial^\alpha\tilde{ u}{\rm d}x
-
\varepsilon
\int_{\mathbb R^3}
\partial^\alpha\tilde{ F}_1
\operatorname{div}\partial^\alpha\tilde{ u}{\rm d}x .
\end{align*}
Using Lemma \ref{L2.1}, we obtain
\begin{align}\label{G6.30}
\tilde{ K}_1
\leq&
-\varepsilon\frac{{\rm d}}{{\rm d}t}
\int_{\mathbb R^3}
\nabla\partial^\alpha\tilde{ q}\cdot
\partial^\alpha\tilde{ u}{\rm d}x
+
C\|\nabla\tilde{ u}\|_{H^1}^2
+
C\delta_0\|\nabla\tilde{ q}\|_{H^1}^2
\nonumber\\
&+
C\varepsilon^2
\|\nabla D_t^u\pi_1\|_{H^1}^2
+
C\varepsilon^2
\|\nabla\pi_1\|_{H^2}^2
\|\tilde{ u}\|_{H^2}^2 .
\end{align}
For the terms \(\tilde{K}_2, \dots,\tilde{ K}_6\),  we use Lemmas \ref{L2.1} and
\ref{L2.4} to deduce that
\begin{align}\label{G6.31}
\sum_{j=2}^{6}\tilde{ K}_j
\leq 
\frac{1}{3}\|\nabla\partial^\alpha\tilde{ q}\|_{L^2}^2
+
C\|\nabla\tilde{ u}\|_{H^2}^2
+
C\|\tilde{ u}-\tilde{\omega}\|_{H^1}^2
+
C\delta_0\|\nabla\tilde{q}\|_{H^1}^2 .
\end{align}
For  \(\tilde{ K}_7\), by Young's inequality
we have
\begin{align}\label{G6.32}
|\tilde{ K}_7|
\leq\,&
\frac{1}{3}\|\nabla\partial^\alpha\tilde{ q}\|_{L^2}^2
+
C\|\nabla\tilde{ u}\|_{H^1}^2
+
C\|\tilde{ u}-\tilde{\omega}\|_{H^1}^2\nonumber\\
&+
C\big(\delta_0+\delta_1\big)
\big(
\|\nabla\tilde{ u}\|_{H^1}^2
+
\|\tilde{ u}-\tilde{\omega}\|_{H^1}^2
\big)
\nonumber\\
&+
C\varepsilon^2
\big(
\|\nabla(\pi_1,\pi_2)\|_{H^2}^2
+
\|\nabla u^\varepsilon\|_{H^3}^2
+
\|\omega-u\|_{H^2}^2
\big).
\end{align}
Then, putting \eqref{G6.30}--\eqref{G6.32} into \eqref{G6.28} and summing over \(|\alpha|\leq1\), we obtain 
\begin{align}\label{G6.33}
&\quad \varepsilon\frac{{\rm d}}{{\rm d}t}
\sum_{|\alpha|\leq1}
\int_{\mathbb R^3}
\nabla\partial^\alpha\tilde{ q}\cdot
\partial^\alpha\tilde{u}{\rm d}x
+
\widetilde\eta_1\|\nabla\tilde{ q}\|_{H^1}^2
\nonumber\\
&\qquad \lesssim
\|\nabla\tilde{ u}\|_{H^2}^2
+
\|\tilde{ u}-\tilde{\omega}\|_{H^1}^2
+
C\delta_0\|\nabla\tilde{ q}\|_{H^1}^2+
C\big(\delta_0+\delta_1\big)
\big(
\|\nabla\tilde{ u}\|_{H^1}^2
+
\|\tilde{ u}-\tilde{\omega}\|_{H^1}^2
\big)
\nonumber\\
&\qquad \quad +
C\varepsilon^2
\big(
\|\nabla D_t^u\pi_1\|_{H^1}^2
+
\|\nabla(\pi_1,\pi_2)\|_{H^2}^2
+
\|\nabla u^\varepsilon\|_{H^3}^2
+
\|\omega-u\|_{H^2}^2
\big),
\end{align}
for some constant $\widetilde \eta_1>0$. Similarly, we   obtain the following estimates on $\tilde{r}$:
\begin{align}\label{G6.38}
&\varepsilon\frac{{\rm d}}{{\rm d}t}
\sum_{|\alpha|\leq1}
\int_{\mathbb R^3}
\nabla\partial^\alpha\tilde{ r}\cdot
\partial^\alpha\tilde{\omega}{\rm d}x
+
\widetilde{\eta_2}\|\nabla\tilde{ r}\|_{H^1}^2
\nonumber\\
&\quad \lesssim
\|\nabla\tilde{ u}\|_{H^2}^2
+
\|\tilde{ u}-\tilde{\omega}\|_{H^1}^2
+
C\delta_0|\nabla\tilde{r}\|_{H^1}^2
+
C\delta_1
\big(
\|\nabla\tilde{ u}\|_{H^1}^2
+
\|\tilde{u}-\tilde{\omega}\|_{H^2}^2
\big)\nonumber\\
&\qquad 
+
C\varepsilon^2
\big(
\|\nabla D_t^\omega\pi_2\|_{H^1}^2
+
\|\nabla\pi_2\|_{H^2}^2
\big),
\end{align}
for some constant $\widetilde \eta_{2}>0$.

Thus, combining \eqref{G6.33} and \eqref{G6.38},   and integrating over \([0,t]\), we obtain
\begin{align}\label{G6.40}
&\int_0^t
\|\nabla(\tilde{ q},\tilde{ r})(\tau)\|_{H^1}^2{\rm d}\tau\nonumber\\
&\quad \lesssim
\varepsilon
\sup_{0\leq \tau\leq t}
\|(\tilde{q},\tilde{ u},\tilde{r},\tilde{\omega})(\tau)\|_{H^2}^2
+
\int_0^t
\big(
\|\nabla\tilde{u}(\tau)\|_{H^2}^2
+
\|\tilde{ u}-\tilde{\omega}\|_{H^2}^2
\big){\rm d}\tau+
C\big(\delta_0+\delta_1\big)
\tilde{\mathcal X}(t)
\nonumber\\
&\qquad +
C\varepsilon^2
\int_0^t
\big(
\|\nabla(D_t^u\pi_1,D_t^\omega\pi_2)\|_{H^1}^2
+
\|\nabla(\pi_1,\pi_2)\|_{H^2}^2
+
\|\nabla u^\varepsilon\|_{H^3}^2
+
\|\omega-u\|_{H^2}^2
\big){\rm d}\tau .
\end{align}
By Theorems \ref{Th1} and \ref{Th2}, Proposition \ref{Ppressure}, \eqref{G6.3} and \eqref{G6.14}, we obtain
\begin{align}
&\sup_{0\leq \tau\leq t}
\|(\tilde{q},\tilde{ u},\tilde{r},\tilde{\omega})(\tau)\|_{H^2}^2
+
\int_0^t
\big(
\|\nabla\tilde{u}(\tau)\|_{H^2}^2
+
\|\tilde{ u}-\tilde{\omega}\|_{H^2}^2
\big){\rm d}\tau\nonumber\\
&\qquad\lesssim C\tilde{\mathcal X}(0)
+
C\varepsilon^2
+
C\big(\delta_0+\delta_1
\big)
\tilde{\mathcal X}(t),\label{G6.40-1}
\end{align}
and
\begin{align}
&
\int_0^t
\big(
\|\nabla(D_t^u\pi_1,D_t^\omega\pi_2)\|_{H^1}^2
+
\|\nabla(\pi_1,\pi_2)\|_{H^2}^2
+
\|\nabla u^\varepsilon\|_{H^3}^2
+
\|\omega-u\|_{H^2}^2
\big){\rm d}\tau
\lesssim
1.\label{G6.40-2}
\end{align}
Substituting \eqref{G6.40-1} and \eqref{G6.40-2} into \eqref{G6.40}, we obtain \eqref{G6.27}. The proof of Lemma \ref{L6.3} is complete.
\end{proof}

With Lemmas \ref{L6.1}--\ref{L6.3} in hand, we  now prove
Theorem \ref{Th3}.

\begin{proof}[Proof of Theorem \ref{Th3}]
Combining the estimates \eqref{G6.3},  
\eqref{G6.14}  and \eqref{G6.27}, we obtain
\begin{align}\label{G6.41}
\tilde{\mathcal X}(t)
\leq
C\tilde{\mathcal X}(0)
+
C\varepsilon^2
+
C\big(
\delta_0+\delta_1
\big)
\tilde{\mathcal X}(t),
\end{align}
for all \(t\geq0\), where \(C>0\) is a constant independent of \(\varepsilon\).
By the well-prepared initial condition \eqref{TD1}, the initial error variables satisfy
\begin{align*}
\tilde{\mathcal X}(0)
\leq 
C\|(\tilde{ q}_0,\tilde{ u}_0,\tilde{ r}_0,\tilde{\omega}_0)\|_{H^2}^2
\leq C\varepsilon^2.
\end{align*}
Now we choose  \(\delta_0>0\) and \(\delta_1>0\) sufficiently small such that
\begin{align*}
C\big(\delta_0+\delta_1
\big)
\leq
\frac12.
\end{align*}
The last term on the right-hand side of \eqref{G6.41} can be absorbed into the left-hand side. Thus, we obtain 
\begin{align}
\tilde{X}(t)
\leq
C\varepsilon^2,
\quad \text{for any}\quad
t\geq0.\nonumber
\end{align}
Therefore, \eqref{TD2} and \eqref{convergence1} follow. The proof of Theorem \ref{Th3} is complete.
\end{proof}

\bigskip 
\noindent\textbf{Acknowledgements.}
F. C. Li and J. K. Ni were supported by NSFC (Grant No. 12331007). F. C. Li was also supported by the ``333 Project" of Jiangsu Province. Z. P. Zhang  was supported by NSFC (Grant No. 12471215) and Taishan Scholars Program (tsqn202507101). Z. Zhang was supported by the General Research Fund (Project No. 15300225) from Hong Kong RGC.

\vspace{2mm}

\noindent\textbf{Conflict of interest.} The authors declare no conflicts of interest.

\vspace{2mm}

\noindent\textbf{Data availability statement.}
No dataset was generated or analyzed during the current study.

\bibliographystyle{plain}

\begin{thebibliography}{aaa}


%\bibitem{AF-Pa-2003} R. A. Adams, J. J. F. Fournier,
%{\it Sobolev Spaces}, second edition. Elsevier/Academic Press, Amsterdam, 2003.

\bibitem{AlazardARMA2006}
T. Alazard, Low Mach number limit of the full Navier--Stokes equations, {\it Arch. Ration. Mech. Anal.} {\bf 180 (1)} (2006)  1--73.

\bibitem{BBBDLLT-irma-2005} C. Baranger, G. Baudin, L. Boudin, 
B. Despr\'{e}s, F. Lagouti\`{e}re, E. Lapébie, T. Takahashi,
Liquid jet generation and break-up, 
Numerical methods for hyperbolic and kinetic problems, 
{\it IRMA Lect. Math. Theor. Phys.} {\bf 7} (2005) 149--176.

\bibitem{BBJM-esaim-2005} C. Baranger, L. Boudin, P. Jabin, S. Mancini,
A modeling of biospray for the upper airways, CEMRACS 2004 --- {\it Mathematics and Applications to Biology and Medicine}, 41--47, ESAIM Proc.,  14, EDP Sciences, Les Ulis, 2005.

\bibitem{BWC-zamm-2000} R. Bürger, W. Wendland, F. Concha, 
Model equations for gravitational 
sedimentation-consolidation processes, {\it ZAMM Z. Angew. Math. Mech.} {\bf 80} (2000) 79--92.


\bibitem{CCK-poincare-2016} J. A. Carrillo, Y.-P. Choi,  T. Karper,  
On the analysis of a coupled kinetic-fluid model with local alignment forces,
{\it Ann. Inst. H. Poincar\'{e} C Anal. Non Lin\'{e}aire} {\bf 33 (2)} (2016)  273--307.

\bibitem{CDM-KRM-2011}
J. A. Carrillo, R. Duan,  A. Moussa,  
Global classical solutions close to equilibrium to the Vlasov--Fokker--Planck--Euler system,
{\it Kinet. Relat. Models} {\bf 4 (1)} (2011)  227--258.

\bibitem{Choi-SIMA-2016} Y.-P. Choi,  
Global classical solutions and large-time behavior of the two-phase fluid model,
{\it SIAM J. Math. Anal.} {\bf 48 (5)} (2016) 3090--3122.

\bibitem{CJ-M3AS-2021} Y.-P. Choi, J. Jung,  
Asymptotic analysis for a Vlasov--Fokker--Planck/Navier--Stokes system in a bounded domain,
{\it Math. Models Methods Appl. Sci.} {\bf 31 (11)} (2021)   2213--2295.

\begin{comment}
\bibitem{Danchin-IM-00} R. Danchin,  
Global existence in critical spaces for compressible Navier--Stokes equations,
{\it Invent. Math.} {\bf 141 (3)} (2000)  579--614.
\end{comment}

\bibitem{Danchin-2002} R. Danchin, 
Zero Mach number limit in critical spaces for compressible Navier--Stokes equations,
{\it Ann. Sci. \'{E}cole Norm. Sup. (4)} {\bf 35 (1)} (2002) 27--75.

\bibitem{Dk-MZ-1992} K. Deckelnick, 
Decay estimates for the compressible Navier--Stokes equations in unbounded domains,
{\it Math. Z.}   
{\bf 209 (1)} (1992) 115--130.

 

\bibitem{Ebin-CPAM-1979}D. G. Ebin, 
The initial-boundary value problem for subsonic fluid motion,
{\it Comm. Pure Appl. Math.} {\bf 32 (1)} (1979)  1--19.

\begin{comment}
\bibitem{FNP-JMFM-2001} E. Feireisl,  A. Novotn\'{y}, H. Petzeltov\'{a}, 
On the existence of globally defined weak solutions to the Navier--Stokes equations,
{\it J. Math. Fluid Mech.} {\bf 3 (4)} (2001) 358--392.
\end{comment}

\bibitem{EWW-ARMA-2016} S. Evje,  W. Wang,  H. Wen,  
Global well-posedness and decay rates of strong solutions to a non-conservative compressible two-fluid model,
{\it Arch. Ration. Mech. Anal.} {\bf 221 (3)} (2016) 1285--1316.

\bibitem{EW-SIMA-2015} S. Evje,  H. Wen,  
Global solutions of a viscous gas-liquid model with unequal fluid velocities in a closed conduit,
{\it SIAM J. Math. Anal.} {\bf 47 (1)} (2015)  381--406.


\bibitem{HJ-NARWA-2025} H. Hong, K. Jong,  
Low Mach number limit for the compressible Euler--Navier--Stokes two-phase flow model in $\mathbb R^3$,
{\it Nonlinear Anal. Real World Appl.} {\bf 84} (2025) Paper No. 104267.

\begin{comment}
\bibitem{JLN-2025} P. Jiang, F. Li, J. Ni, Global existence and short/long time behavior of classical solutions to the incompressible inhomogeneous Navier--Stokes--Vlasov--Fokker--Planck system,   {\it J. Math. Phys},   {\bf   66 (12)}  (2025) Paper No. 121509.
\end{comment}

\bibitem {commutator1} T. Kato, G. Ponce,  Commutator estimates and the Euler and Navier–Stokes equations,  {\it Commun. Pure Appl. Math.} {\bf 41 (7)} (1988) 891--907.


\bibitem {commutator2} C. E. Kenig, G. Ponce, L. Vega, Well-posedness and scattering results for the generalized Korteweg--de Vries equation via the contraction principle, {\it J. Amer. Math. Soc.} {\bf 4 (2)} (1991) 323--347.

\bibitem{KM-CPAM-1981}
S. Klainerman, A. Majda, Singular limits of quasilinear hyperbolic systems with large parameters and the incompressible limit of compressible fluids, {\it Comm. Pure Appl. Math.} 
{\bf 34} (1981)  481--524.

\bibitem{KM-CPAM-1982}
S. Klainerman, A. Majda,
Compressible and incompressible fluids, {\it Comm. Pure Appl. Math.} {\bf 35}
(1982)  629--651.

\begin{comment}
\bibitem{LNZW-2026} F. Li, J. Ni, D.-H. Wang, Z. Zhang,  Low Mach number limit of the compressible Euler--Vlasov--Fokker--Planck system with well-prepared initial data in $\mathbb{R}^3$, preprint.

\end{comment}

\bibitem{LS-SIMA-2023} H.-L. Li,  L.-Y. Shou,  
Global existence and optimal time-decay rates of the compressible Navier--Stokes--Euler system,
{\it SIAM J. Math. Anal.} {\bf 55 (3)} (2023)  1810--1846.

\bibitem{LM-JMPA-1998} P.-L. Lions, N. Masmoudi,  
Incompressible limit for a viscous compressible fluid,
{\it J. Math. Pures Appl. (9)} {\bf 77 (6)} (1998) 585--627.


 \bibitem{Matsumura-Nishida-1979}
A. Matsumura, T. Nishida,
The initial value problem for the equations of motion of compressible viscous and heat-conductive fluids,
 {\it Proc. Japan Acad. Ser. A Math. Sci.} {\bf 55 (9)} (1979)  337--342.

\bibitem{Matsumura-Nishida-1980}
A. Matsumura, T. Nishida,
The initial value problem for the equations of motion of viscous and heat-conductive gases,
{\it J. Math. Kyoto Univ.}  {\bf 20 (1)} (1980) 67--104.

\bibitem{MV08} A. Mellet, A. Vasseur, 
Asymptotic analysis for a Vlasov--Fokker--Planck/compressible Navier--Stokes system of equations,
{\it Comm. Math. Phys.} {\bf 281 (3)} (2008) 573--596.

 \bibitem{MS-ARMA-2001}
G. M\'{e}tivier, S. Schochet, The incompressible limit of the non-isentropic Euler equations, {\it  Arch. Ration. Mech. Anal.}  {\bf 158} (2001) 61--90.

\bibitem{RM-1952} W. E. Ranz, W. R. Marshall, Evaporation from drops, part I, Chem. Eng. Prog. {\bf 48} (1952) 141--146.

\bibitem{RM-1952-a} W. E. Ranz, W. R. Marshall, Evaporation from drops, part II, Chem. Eng. Prog. {\bf 48}
(1952) 173--180.




\bibitem{Schochet-CMP-1986} S. Schochet,  
The compressible Euler equations in a bounded domain: existence of solutions and the incompressible limit,
{\it Comm. Math. Phys.} {\bf 104 (1)} (1986)  49--75.

\bibitem{Schochet-JDE-1994} S. Schochet, 
Fast singular limits of hyperbolic PDEs,
{\it J. Differential Equations} {\bf 114 (2)} (1994)  476--512.

\begin{comment}
\bibitem{SK-HMJ-1985} Y. Shizuta, S. Kawashima, 
Systems of equations of hyperbolic-parabolic type with applications to the discrete Boltzmann equation,
{\it Hokkaido Math. J.} {\bf 14 (2)} (1985) 249--275.
\end{comment}

\bibitem{Stein-1970} E. M. Stein, {\it Singular Integrals and Differentiability Properties of Functions,} Princeton, NJ: Princeton University Press, 1970.


 

\bibitem{TZ-JMAA-2021} H. Tang,  Y. Zhang,  
Large time behavior of solutions to a two phase fluid model in $\mathbb R^3$,
{\it J. Math. Anal. Appl.} {\bf 503 (2)} (2021)  Paper No. 125296.
 

\bibitem{Ukai-JMKU-1986}
S. Ukai, The incompressible limit and the initial layer of the compressible Euler equation. {\it J. Math. Kyoto Univ.} {\bf 26} (1986) 323--331.

\bibitem{Wfa-1958} F.-A. Williams, Spray combustion and atomization, {\it Phys. Fluids} {\bf 1} (1958) 541--555.

\bibitem{Wfa-1985} F.-A. Williams, {\it Combustion Theory}, Benjamin Cummings, 1985.

\bibitem{WZZ-SIMA-2020} G. Wu,  Y. Zhang,  L. Zhou,  
Optimal large-time behavior of the two-phase fluid model in the whole space,
{\it SIAM J. Math. Anal.} {\bf 52 (6)} (2020) 5748--5774.

\bibitem{WZT-M2AS-2023} Y. Wu,  Y. Zhang, H. Tang,  
Optimal decay rate of solutions to the two-phase flow model,
{\it Math. Methods Appl. Sci.} {\bf 46 (2)} (2023)  2538--2568.



\bibitem{YCZ-SIMA-2012} L. Yao,  C. Zhu,  R. Zi,  
Incompressible limit of viscous liquid-gas two-phase flow model,
{\it SIAM J. Math. Anal.} {\bf 44 (5)}(2012)  3324--3345.

\bibitem{ZWXM-ZAMP-2021} Y. Zhang,  J. Wang,  C. Xiao,  L. Ma,  
Global existence and time decay rates of the two-phase fluid system in $R^3$,
{\it Z. Angew. Math. Phys.} {\bf 72 (5)} (2021)  Paper No. 180.
 







\end{thebibliography}

\end{document}